\documentclass[12pt]{amsart}
\usepackage{amssymb}
\usepackage{amsmath}
\usepackage{version}
\usepackage{color}
\usepackage{amsfonts}
\numberwithin{equation}{section}
\newtheorem{theorem}{Theorem}[section]
\newtheorem{lemma}[theorem]{Lemma}
\newtheorem{corollary}[theorem]{Corollary}
\newtheorem{example}[theorem]{Example}

\newtheorem{definition}[theorem]{Definition}
\newtheorem{proposition}[theorem]{Proposition}

\newtheorem{remark}[theorem]{Remark}

\let\cal\mathcal
\def\A{\mathcal A}  % A - Banach algebra
\def\B{\mathcal B}
\def\K{\mathcal K}

\def\M{\mathcal M}
\def\N{\mathcal N}

\def\D{\mathcal D}
\def\U{\mathcal U}
\def\V{\mathcal V}
\def\W{\mathcal W}

\def\H{\mathcal H}  % H for Hochschild homology or cohomology
\def\B{\mathcal B}  % B for coboundaries
\def\NN{\mathbf N} % N for natural numbers
\def\CC{\mathbf C} % C for complexes
\def\RR{\mathbf R} % R for reals
\newcommand\beq{\begin{equation}}

\newcommand\eeq{\end{equation}}
\newcommand\id{\rm id}

\numberwithin{equation}{section}
\begin{document}

\title{Projectivity of Banach and $C^*$-algebras of  continuous fields}

\author{David Cushing and Zinaida A. Lykova}

\date{26 April 2011}

\begin{abstract} We give necessary  and sufficient conditions for the left projectivity and biprojectivity of Banach algebras defined by locally trivial continuous fields of Banach algebras. We identify projective  $C^*$-algebras  $\A$ defined by locally trivial continuous fields $\mathcal{U} = \{\Omega,(A_t)_{t \in \Omega},\Theta\}$ such that each $C^*$-algebra $ A_{t}$ has a strictly positive element. For  a commutative
 $C^*$-algebra $\D$ contained in ${\cal B}(H)$, where $H$ is a separable 
Hilbert space, we show that the condition of left projectivity of $\D$ is  equivalent to the existence of a strictly positive element in $\D$ and so to the spectrum of $\D$ being a Lindel$\ddot{\rm o}$f  space.
\end{abstract}

\thanks {David Cushing is partially supported by an EPSRC DTA grant.}


\maketitle

\markboth{David Cushing and Zinaida A. Lykova}
{Projectivity of  Banach and $C^*$-algebras of  continuous fields}

\section{Introduction}  

The study  of projective, injective and flat 
modules over Banach algebras and operator algebras  has  attracted  quite a number mathematicians. Some 
recent papers on the projectivity and  injectivity  of Banach modules  and on their applications are \cite{Ari00, BDL, DDPR, DP, MJo, Wh, Se00, Ra}. 
The identification of projective
algebras and projective closed ideals of Banach algebras,
 besides being of independent interest, is closely connected to  continuous Hochschild cohomology. 
 One of the main methods for computing  cohomology groups
is to construct  projective or injective resolutions of 
the corresponding module and the algebra. In this paper we consider the question of the left projectivity and biprojectivity of some Banach algebras $\A$ and we give applications to the second continuous Hochschild  cohomology group 
$\H^2(\A, X)$ of  $\A$ and to the strong splittability of singular extensions of $\A$.

This paper concerns those $C^*$-algebras of  continuous fields that are 
left projective. The basic examples of  $C^*$-algebras defined by continuous fields are the following: the $C^*$-algebra $C_0(\Omega,\A)$ of the constant field over  a Hausdorff locally compact space $\Omega$ defined by  a $C^*$-algebra $\A$, and the direct sum of the $C^*$-algebras $(A_{\lambda})_{\lambda \in \Lambda}$, where  $\Lambda$ has  the discrete topology.
Continuous fields of $C^*$-algebras were found useful in the characterisation of exactness and nuclearity of $C^*$-algebras \cite{KirWas}.
As to the left projectivity of $C^*$-algebras, the following results are known. All $C^*$-algebras with strictly positive elements are left and right projective, and so all separable $C^*$-algebras are hereditarily projective, see \cite[Theorem 2.5]{PhR} and \cite[Theorem 1]{Ly2}.  No infinite-dimensional $AW^*$-algebra is  hereditarily projective \cite{Ly2,Ly02}. Thus any infinite-dimensional von Neumann algebra is not hereditarily projective. Recall that a Banach algebra $\A$  is {\it hereditarily projective} if every closed left ideal of $\A$ is projective in the category of left Banach $\A$-modules. 
A complete description of hereditarily projective  commutative $C^*$-algebras $C(\Omega)$ as algebras having a hereditarily paracompact spectrum is given in \cite{He3}. It is quite difficult to get a  complete description of left projective noncommutative $C^*$-algebras because of the richness of   $C^*$-algebras. On the other hand, very broad classes of $C^*$-algebras can be obtained as  $C^*$-algebras defined by continuous fields of very simple $C^*$-algebras. For example, by \cite[Theorems 10.5.2 and 10.5.4]{Di}, there exists a canonical bijective correspondence between the liminal  $C^*$-algebras with Hausdorff spectrum $\Omega$ and the continuous fields of non-zero elementary $C^*$-algebras over  $\Omega$.

In this paper we identify left 
projective  $C^*$-algebras  $\A$ defined by locally trivial continuous fields $\mathcal{U} = \{\Omega,(A_t)_{t \in \Omega},\Theta\}$
 where the topological dimension $\dim \Omega \le \ell$ and each $ A_{t}$ has a strictly positive element (Theorem \ref{Ax-strictly-positive-h-H2}). 
We prove also that, for a commutative $C^*$-algebra $\D = C_0(\Omega)$ for which $\Omega$ has countable Suslin number,   the condition of left projectivity is  equivalent to the existence of a strictly positive element (Theorem \ref{count_Suslin_number}), but not to the separability of $\D$  (Example \ref{nonsepA}). 

In Section \ref{definitions} we recall the classical definition of a continuous  field of Banach and $C^*$-algebras \cite[Chapter 10]{Di}. We also recall notation and terminology used in topology and in the homological theory of Banach algebras. 
In Section \ref{nec_cond_proj} we prove  that if 
$\mathcal{U} = \{\Omega,(A_t)_{t \in \Omega},\Theta\}$ is a locally trivial continuous field of Banach algebras such that the Banach algebra $\A$ defined by $\mathcal{U}$ is projective in $\A$-$\mathrm{mod}$ ($\mathrm{mod}$-$\A$, $\A$-$\mathrm{mod}$-$\A$), then the  Banach algebras $A_{x}$, $x \in \Omega$, are uniformly left (right, bi-)projective.
In Section \ref{A-proj-Omega-paracompact} we consider  the situation that
 $\Omega$ is a disjoint union of a family of open subsets $\{ W_\mu \}$, $\mu \in \M$, of $\Omega$. In this case we
 say that {\em $\mathcal{U} = \{\Omega,(A_t)_{t \in \Omega},\Theta\}$ is a {\it disjoint union} of $\mathcal{U}|_{W_{\mu}}$}, $\mu \in \M$.
We show that if $\A$ is defined by such a $\mathcal{U}$, then 
the left projectivity of $\A$ implies that $\Omega$ is paracompact. In Section \ref{biprojective}
we prove that biprojectivity of the Banach algebra $\A$ defined by  a locally trivial continuous field of Banach algebras implies that $\Omega$ is discrete. We give a description of contractible Banach algebras defined by  locally trivial continuous fields of Banach algebras.
In Section \ref{left-proj-C*-posit-elem} we give six different criteria for  a commutative
 $C^*$-algebra  contained in ${\cal B}(H)$, where $H$ is a separable 
Hilbert space, to be left projective (Corollory \ref{Hseparab_AinB(H)}).
In the noncommutative case we also show that, for 
a Hausdorff locally compact space $\Omega$ with $\dim\Omega \le \ell$,  for some $\ell \in \NN$, and for a locally  trivial continuous field $\mathcal{U} = \{\Omega,(A_t)_{t \in \Omega},\Theta\}$ of $C^*$-algebras $A_t$, $t \in \Omega$, with strictly positive elements, the following conditions are equivalent: {\rm (i)} $\Omega$ is paracompact and  {\rm (ii)} $\A$ defined by $\mathcal{U}$ is left projective and $\mathcal{U}$ is a disjoint union of $\sigma$-locally  trivial continuous fields of $C^*$-algebras (Theorem \ref{Ax-strictly-positive-h-H2}).

%%%%%%%%%%%%%%%%%%%%%%

\section{Definitions and notation}\label{definitions}

  We recall some notation and terminology used in topology \cite{En} and  in  the homological theory of Banach algebras \cite{He4}.

 For any  Banach algebra $\A$, not necessarily unital,
$\A_+$ is the Banach algebra obtained by adjoining an identity $e$ to $A$.
We denote the projective tensor product of Banach spaces 
by $ \widehat{\otimes} $.  The category of  left  Banach  $\A$-modules is denoted by $\A$-$\mathrm{mod}$, the category of  right  Banach  $\A$-modules is denoted by $\mathrm{mod}$-$\A$ and the category of Banach  $\A$-bimodules is denoted by  $\A$-$\mathrm{mod}$-$\A$. 

A Banach $\A$-bimodule $X$ is  {\it right annihilator} if $X \cdot \A =
\{x \cdot a: a \in \A, \; x \in X \} = \{0 \}$.
For a Banach space $E$, we will denote by  $E^*$ 
 the dual space of $E$. For a Banach $\A$-bimodule $X$, $X^*$ is the Banach $\A$-bimodule dual to $X$ with the module multiplications given by
\[ (a \cdot f)(x) = f(x \cdot a), \; (f \cdot a)(x) = f(a \cdot x)\;\;(a \in \A, \; f \in X^*, \; x \in X ). 
\]

Let ${\cal K}$ be one of the above categories of Banach 
modules and morphisms. If $X, Y$ are objects of ${\cal K}$, 
the Banach space of morphisms from $X$ to $Y$ is denoted by 
$h_{\cal K} (X, Y)$.
 A module $P$ in ${\cal K}$ is called 
{\it projective in ${\cal K}$} if, for
each module $Y$ in ${\cal K}$ and each  epimorphism of
modules $\varphi \in h_{\cal K}(Y, P)$ such that $\varphi$ has a right
inverse as a morphism of Banach spaces, there exists a 
morphism $\psi \in h_{\cal K}(P, Y)$ which is   a right
inverse of $\varphi$.

 In  the category of  left  Banach  $\A$-modules, let
$X$ be a left  Banach  $\A$-module and let us consider
 the canonical morphism 
$\pi_{X}:  \A_+ \widehat{\otimes} X \rightarrow X: a \otimes x \mapsto a \cdot x.$ Then  $X$ is  projective
if and only if there is a morphism of left  Banach $\A$-modules 
 $\rho:  X \rightarrow \A_+ \widehat{\otimes} X$ 
such that $\pi_{X} \circ \rho = {\rm id}_{X} $.
 Throughout the paper  ${\rm id}$  denotes the identity operator.
 We say that
a Banach algebra $\A$ is {\it left (right) projective} if it is  projective in  the category of  left (right) Banach  $\A$-modules.
 Every unital Banach algebra is left and right projective. We say that
a Banach algebra $\A$ is {\it biprojective} if it is a projective
 Banach $\A$-bimodule.

We say that the Banach algebras $A_{x}$, $x \in \Omega$, are {\it uniformly left projective} if, for every $x \in \Omega$,
 there is a morphism of left Banach $A_{x}$-modules
\[ \rho_{A_{x}}: A_{x} \to (A_{x})_+ \widehat{\otimes} A_{x}
\]
such that $\pi_{A_{x}} \circ \rho_{A_{x}} = {\rm id}_{A_{x}} $ and 
$\sup_{x \in \Omega} \| \rho_{A_{x}} \| < \infty$.

 For  a Banach algebra  $\A$ and   a Banach $\A$-bimodule $X$,  
we define an {\it  $n$-cochain} to be a bounded $n$-linear 
operator from $\A \times \dots \times \A$ into $X$ and we denote 
 the space of $n$-cochains by $ C^n(\A,X)$. For $n=0$ the space  
$ C^0(\A,X)$ is
defined to be $X$. Let us consider the {\it standard cohomological complex}

\vspace*{0.2cm}
\hspace{0.2cm}
$ 0 \longrightarrow C^0(\A,X) \stackrel { \delta^0} { \longrightarrow} 
\dots \longrightarrow C^n(\A,X) \stackrel { \delta^n} { \longrightarrow}  
 C^{n+1}(\A,X) 
\longrightarrow \dots,
\hfill {({\cal C}^{\sim}(\A,X))} $

\vspace*{0.2cm}
\noindent where the coboundary operator $ \delta^n $ is defined by

\vspace*{0.2cm}
\begin{center}
$(\delta^n f)(a_1,...,a_{n+1}) = a_1 \cdot f(a_2,...,a_{n+1}) +$
\end{center}
\begin{center}
$ \sum_{i=1}^n (-1)^i f(a_1,...,a_i a_{i+1},...,a_{n+1}) +
(-1)^{n+1} f(a_1,...,a_n) \cdot a_{n+1}.$
\end{center}

\vspace*{0.2cm}
\noindent 
The $n$th cohomology group of the standard cohomology complex ${\cal C}^{\sim}(\A,X)$, denoted by ${\H}^n(\A,X)$, is called the {\it  $n$th continuous Hochschild
cohomology group of $\A$ with coefficients in $X$} (see, for example, 
\cite{Jo1} or \cite[Definition I.3.2]{He4}). It is a complete seminormed space.
Definitions and results on the strong and algebraic splittings of extensions of Banach algebras can be found in \cite{BDL}.

 For a Banach algebra $\A$ and for a subset $Y$ of a left Banach $\A$-module $X$, $\A Y$ is 
the linear span of the set $\A \cdot Y =\{a \cdot y: a \in \A, y \in Y\}$ and  $\overline{\A Y}$ is 
the closure of $\A Y$. For a Banach algebra $\A$ with bounded approximate identity, by an extension of the Cohen factorization 
theorem,  for every left Banach  $\A$-module
$X$, $\overline{\A X} = \A X = \A \cdot X =\{a \cdot x: a \in \A, x \in X\}$ 
\cite{Hew, CF-T}.

  Let  $\B$ be a commutative $C^*$-algebra
 and let $I$ be a closed ideal of $\B$. In view of \cite[1.4.1]{Di},
 $\B$ is isomorphic to the $C^*$-algebra $C_0(\widehat{\B})$
of  continuous functions on $\widehat{\B}$ vanishing at infinity. Here 
$\widehat{\B}$ is the space of  characters of $\B$ with the relative 
weak-* topology from the dual space of $\B$, 
called the 
{\it spectrum } of $\B$. We recall \cite{Ri} that the spectrum of a 
closed ideal $I$
of $\B$ is homeomorphic to the open set 
$\widehat{I} = \{t \in \widehat{\B}: x(t)\neq  0 $ for some $ x \in I\}$.

Let $\Omega$ be a  topological space, and let $n$ be a positive integer. We recall that the {\it order} (or {\it multiplicity}) of a cover $\W$ of  $\Omega$ is the maximum number $n$ such that every $x \in \Omega$ is covered by no more than $n$ elements of $\W$.
For a normal topological space $\Omega$, we
 say that the {\it topological dimension} of $\Omega$ is less than or equal to $n$ if the following condition is satisfied:
 every locally finite open cover of $\Omega$
possesses an open locally finite refinement  of order $n$ (Dowker's Theorem \cite[Theorem 7.2.4]{En}). We write $\dim \Omega \le n$.

The $C^*$-algebra of all complex $n \times n$ matrices is
denoted by $M_n({\CC})$. For a Hilbert space $H$, $\B(H)$ and $\K(H)$ will denote the  $C^*$-algebras of all continuous and all compact 
linear operators on $H$ respectively.

\subsection{Continuous fields  of Banach and $C^*$-algebras}

We shall use the classical definition of a continuous field of 
Banach or $C^*$-algebras over a locally compact Hausdorff space $\Omega$ \cite[Chapter 10]{Di}. 

Let $\Omega$ be a topological space,  and let
$(E_t)_{t \in \Omega}$ be a family of Banach spaces. Every element of  $\prod_{t \in \Omega} E_t$, that is, every function $x$ defined on $\Omega$ such that $x(t) \in E_t$ for each $t \in \Omega$, is called a {\it vector field over} $\Omega$.

\begin{definition}\label{cont-fields}{\rm \cite[Definitions 10.1.2 and 10.3.1]{Di}} A {\em continuous  field}  $\mathcal{U}$ of Banach algebras {\rm ($C^*$}-algebras{\rm )} is a triple
$\mathcal{U} = \{\Omega,(A_t)_{t \in \Omega},\Theta\}$ where
 $\Omega$ is a locally compact Hausdorff space,  
$(A_t)_{t \in \Omega}$ is a family of Banach algebras {\rm ($C^*$}-algebras{\rm )} and  $\Theta $ is an (involutive) subalgebra of $\prod_{t \in \Omega}A_t$ such that
\begin{enumerate} 
\item[{\rm (i)}] for every $t \in \Omega$, the set of $x(t)$ for $x \in \Theta$ is dense in $A_t$; 
\item[{\rm (ii)}] for every $x \in \Theta$, the function $t \to \|x(t)\|$ is continuous on $\Omega$;
\item[{\rm (iii)}] whenever $ x \in \prod_{t \in \Omega}A_t$ and, for every $t \in \Omega$ and every $\varepsilon >0$, there is an $x' \in \Theta$ such that $\|x(t) -x'(t)\| \le \varepsilon $ throughout some neighbourhood of $t$, it follows that $ x \in \Theta$. 
\end{enumerate} 
The elements of  $\Theta$ are called the {\em continuous vector  fields} of $\mathcal{U}$.
\end{definition}

\begin{definition}\label{isom-cont-fields}{\rm \cite[Definitions 10.1.3 and 10.3.1]{Di}}
Let  $\Omega$ be a locally compact Hausdorff space, and let
\[
\mathcal{U} = \{\Omega,(A_t)_{t \in \Omega},\Theta\} \;\;\;\text{and} \;\;\;
\mathcal{U}' =\{\Omega,(A'_t)_{t \in \Omega},\Theta'\}
\] 
be two continuous  fields of Banach algebras {\rm ($C^*$}-algebras{\rm )} over $\Omega$.
 An isomorphism of $\mathcal{U}$ onto $\mathcal{U}'$ is
 a family $\phi =(\phi_t)_{t \in \Omega}$ such that each $\phi_t$ is an  isometric ($*$-)isomorphism of Banach algebras  
$A_t$ onto $A'_t$ and $\phi$ transforms  $\Theta$ into $\Theta'$.
\end{definition}

%%%%%%%%%%%%%%%%%%
\begin{definition}\label{cont-fields-on-Y}{\rm \cite[10.1.6 and 10.1.7]{Di}} Let $\mathcal{U} = \{\Omega,(A_t)_{t \in \Omega},\Theta\}$ be a
 continuous  field  of Banach algebras over a  locally compact Hausdorff space $\Omega$. Let $Y \subseteq \Omega$ and $ t_0 \in Y$. A vector field $x$ over $Y$ is said to be {\em continuous at $t_0$}
if, for every $\varepsilon >0$, there is an $x' \in \Theta$ such that $\|x(t) -x'(t)\| \le \varepsilon $ throughout some neighbourhood of $t_0$.
A vector field $x$ is said to be {\em continuous on $Y$} if it is continuous at every point of $Y$.

Let $\Theta|_Y$ be the set of continuous vector  fields over $Y$. It is immediate that $ \{Y,(A_t)_{t \in Y},\Theta|_Y\}$ is a continuous  field  of Banach algebras over $Y$, which is called {\em the field induced by $\mathcal{U}$ on $Y$}, and which is denoted by
$\mathcal{U}|_Y$. 
\end{definition}

%%%%%%%%%%%%%%%%%%%%%%

\begin{example}\label{trivial-cont-fields}{\rm \cite[Example 10.1.4]{Di}~
Let $A$ be a Banach algebra ($C^*$-algebra), let  $\Omega$ be a locally compact Hausdorff space, and let $\Theta$ be the (involutive) algebra of continuous mappings of $\Omega$ into 
$A$. For every $t \in \Omega$, put $A_t=A$. Then $\mathcal{U} =\{\Omega, (A_t)_{t \in \Omega},\Theta\}$ is a 
 continuous  field of Banach algebras ($C^*$-algebras) over $\Omega$, called the {\em constant field} over $\Omega$ defined by $A$. A field isomorphic to a constant field is said to be {\em trivial}. If every point of $\Omega$ possesses a neighbourhood $V$ such that $\mathcal{U}|_V$ is trivial, then $\mathcal{U}$ is said to be {\em locally trivial} (\cite[10.1.8]{Di}).
}
\end{example}

\begin{example}\label{discrete-cont-fields}{\rm \cite[Example 10.1.5]{Di}~
If $\Omega$ is discrete in Definition \ref{cont-fields}, then Axioms (i) and (iii) imply that $\Theta$ must be equal to $\prod_{t \in \Omega}A_t$.
}
\end{example}

\begin{definition}\label{cont-fields-C*-alg}
{\rm \cite[Definition 10.4.1]{Di}}~
Let $\mathcal{U}=\{\Omega,(A_t)_{t \in \Omega},\Theta\}$ be a
 continuous  field   of  Banach algebras ($C^*$-algebras) over  a locally compact Hausdorff space $\Omega$.
Let $\A$ be the set  of $x \in \Theta$ such that $\|x(t) \|$ vanishes at infinity on $\Omega$. Then $\A$ with $\|x\|= \sup_{t \in \Omega} \|x(t) \|$ is a  Banach algebra ($C^*$-algebra)
 which we call {\em the Banach algebra ($C^*$-algebra) defined by}
$\mathcal{U}$, or {\em the algebra of sections of} $\mathcal{U}$, or {\em the continuous bundle Banach algebra ($C^*$-algebra)}.
\end{definition}
In this paper we shall use the notation $\mathcal{U} = \{\Omega,A_t,\Theta\}$ instead of $\mathcal{U} = \{\Omega,(A_t)_{t \in \Omega},\Theta\}$. Let $\mathcal{U} = \{\Omega,A_t,\Theta\}$ be a
 continuous  field   of  Banach algebras ($C^*$-algebras) over $\Omega$, and let $\A$ be defined by $\mathcal{U}$.
Then, for every $t \in \Omega$, the map
$
\tau_{t}: \A \to A_t : a \mapsto a(t)
$
is a  Banach algebra homomorphism with dense image ($*$-epimorphism) and $\|\tau_{x}\| \le 1$. We will also consider
$
{\tau_{t}}_+ : \A_+ \to {(A_t)}_+ : a + \lambda e \mapsto a(t) + \lambda e_t
$
where $e$ and $e_t$ are the adjoined identities to $\A$ and $A_t$ respectively.

\section{Necessary conditions for left projectivity of Banach algebras of continuous fields} \label{nec_cond_proj}

The next statement is well-known.

\begin{lemma}\label{A^2neq0} Let  $\A$ be a  Banach algebra such that  $\A \neq \{0\}$ and is projective in $\A$-$\mathrm{mod}$ or in $\mathrm{mod}$-$\A$. Then $\A^2 \neq \{0\}$.
\end{lemma}
\begin{proof} 
We shall prove the statement for $\A$ in the case that $\A$ is left projective. Since $\A$ is left projective, there exists a morphism of left  Banach  $\A$-modules
 $\rho: \A \rightarrow \A_+ \widehat{\otimes} \A$ 
such that $\pi_{\A} \circ \rho = {\rm id}_{\A} $, where 
$\pi_{\A}:  \A_+ \widehat{\otimes} \A \rightarrow \A$ is 
 the canonical morphism.
 By assumption, $\A \neq \{0\}$ and so there exists non-zero element $a \in \A$.  Suppose $\A^2 = \{0\}$; thus the element $a^2$ of $\A$ is trivial and $\rho(a^2) = 0$.

By \cite[Theorem 3.6.4]{Sh}, the element $\rho(a)$ from $\A_+ \widehat{\otimes} \A$ 
can be written as 
$$\rho(a)= \sum_{i=1}^{\infty} \lambda_i~(\alpha_i e + a_i) \otimes b_i,\; \text{ for some} \;
\lambda_i, \alpha_i \in {\mathbb C}, \; a_i, b_i \in \A,\;$$ where 
 $ \sum_{i=1}^{\infty} |\lambda_i| < \infty$ and the sequences
$\{\alpha_i e + a_i\}, \{b_i\}$ converge to zero  in $\A_+$ and $\A$ as $i \to \infty$. Thus the equality $\pi_{\A} \circ \rho = {\rm id}_{\A} $ implies
\[
\pi_{\A} \circ \rho(a) =a \neq 0 
\]
and, because $\A^2 = \{0\}$,
\[
a = \pi_{\A} \left(\sum_{i=1}^{\infty} \lambda_i~(\alpha_i e + a_i) \otimes b_i \right) =
\left(\sum_{i=1}^{\infty} \lambda_i~(\alpha_i e + a_i)  b_i \right)= \sum_{i=1}^{\infty} \lambda_i \alpha_i  b_i.
\]
Note that $\rho$ is a morphism of left  Banach  $\A$-modules; thus
\[
\rho(a^2) =a \left(\sum_{i=1}^{\infty} \lambda_i~(\alpha_i e + a_i) \otimes b_i \right) = \sum_{i=1}^{\infty} \lambda_i~(\alpha_i a + a a_i) \otimes b_i.
\]
Then, since $\A^2 = \{0\}$,  we have
\[
 \rho(a^2) = \sum_{i=1}^{\infty} \lambda_i~\alpha_i a  \otimes b_i = a \otimes \left(\sum_{i=1}^{\infty} \lambda_i \alpha_i  b_i\right)=a  \otimes a \neq 0.
\]
This contradicts  $\rho(a^2) = 0$. Therefore,   $\A^2 \neq \{0\}$.
 \end{proof}

\begin{proposition}\label{PropositionIV.2.8-He4} {\rm \cite[Proposition IV.2.8]{He4}}~ Let $\kappa: \A \to \B $ be a Banach algebra homomorphism with dense image.  Suppose $\A$ has a right bounded approximate identity  $e_\nu$, $\nu \in \Lambda$, and $\A$ is left  projective, that is, there is a  morphism of left Banach $\A$-modules
\[ \rho_{\A}: \A \to \A \widehat{\otimes} \A
\]
such that $\pi_{\A} \circ \rho_{\A} = {\rm id}_{\A} $. Then there is a morphism of left Banach $\B$-modules
\[ \rho_{\B}: \B \to \B \widehat{\otimes} \B
\]
such that $\pi_{\B} \circ \rho_{\B} = {\rm id}_{\B} $ and 
$\| \rho_{\B} \| \le \| \kappa \|^2 \|  \rho_{\A}\|  \sup_{\nu} \| e_\nu \|$, and, in particular, the Banach algebra $\B$ is left projective.
\end{proposition}

\begin{proposition}\label{not-loc-triv-A_t-unif-proj} Let $\Omega$ be a locally compact Hausdorff space, let $\mathcal{U}=\{ \Omega,A_{x}, \Theta \}$ be a continuous field of Banach algebras, and let $\A$ be defined by $\mathcal{U}$. Suppose $\A$ is projective in $\A$-$\mathrm{mod}$ ($\mathrm{mod}$-$\A$, $\A$-$\mathrm{mod}$-$\A$) and  has a right (left, two-sided) bounded approximate identity. Then the Banach algebras $A_{x}$, $x \in \Omega$, are uniformly left (right, bi-)projective.
\end{proposition}

\begin{proof} For every $x \in \Omega$, $\tau_{x}: \A \to A_{x}$ is a  Banach algebra homomorphism with dense image  and $\|\tau_{x}\| \le 1$. By Proposition \ref{PropositionIV.2.8-He4},  there is a morphism of left Banach $A_{x}$-modules
\[ \rho_{A_{x}}: A_{x} \to A_{x} \widehat{\otimes} A_{x}
\]
such that $\pi_{A_{x}} \circ \rho_{A_{x}} = {\rm id}_{A_{x}} $ and 
$\| \rho_{A_{x}} \| \le \|  \rho_{\A}\|  \sup_{\nu} \| e_\nu \|$.
Therefore  the left projectivity of $\A$ implies the uniform left projectivity of the Banach algebras $A_{x}$, $x \in \Omega$.
A similar proof works for  right- and bi- projectivity.
 \end{proof}

%%%%%%%%%%%%%%%%%%%%
\begin{corollary}\label{C*not-loc-triv-A_t-unif-proj} Let $\mathcal{U} = \{\Omega,A_t,\Theta\}$ be a continuous  field  of  $C^*$-algebras over  a locally compact Hausdorff  space 
$\Omega$, and let the $C^*$-algebra $\A$ be defined by $\mathcal{U}$. Suppose $\A$ is projective in $\A$-$\mathrm{mod}$ ($\mathrm{mod}$-$\A$, $\A$-$\mathrm{mod}$-$\A$). Then the $C^*$-algebras $A_{x}$, $x \in \Omega$, are uniformly left (right, bi-)projective.
\end{corollary}

Hence, for the Banach algebra $\A$ defined by a {\it continuous} field  $\mathcal{U} = \{\Omega,A_t,\Theta\}$, 
Proposition \ref{not-loc-triv-A_t-unif-proj} shows that, under the condition of the existence of a bounded approximate identity in $\A$, the projectivity of $\A$ implies the uniform projectivity of the Banach algebras $A_{x}$, $x \in \Omega$. In the next proposition we will show that, for the  Banach algebra $\A$ defined by a {\em locally trivial continuous} field $\mathcal{U} = \{\Omega,A_t,\Theta\}$,
 we do not need the condition  of the existence of a bounded  approximate identity in $\A$ to get  the uniform projectivity of the Banach algebras $A_{x}$, $x \in \Omega$, from the projectivity of $\A$.

\begin{lemma}\label{a-in-Theta} Let  $\Omega$ be a locally compact Hausdorff space, let $\mathcal{U} = \{\Omega,A_t,\Theta\}$ be a locally trivial continuous field of Banach algebras  $\A_{x}$, and let the Banach algebra $\A$ be defined by $\mathcal{U}$.
For every $ a_y \in \A_{y}$ there is $a \in \A$ such that $a(y) = a_y$.
\end{lemma} 
\begin{proof} By \cite[Theorem 5.17]{Ke}, $\Omega$ is regular and, by \cite[Theorem 3.3.1]{En}, $\Omega$ is a Tychonoff space. By assumption, $\mathcal{U}$ is locally trivial and so, for each $y \in \Omega$, there are open neighbourhoods $V_y$ and $U_y$ of $y$ such that $\overline{V_y} \subset U_y$, $\mathcal{U}|_{U_y}$ is trivial and $V_y$ is relatively compact. For each $y \in \Omega$, fix a continuous function $f_y\in C_{0}(\Omega)$ such that $0\leq f_y \leq 1$, $f_y(y)=1$ and $f_y|_{\Omega\setminus U_{y}}=0$. 
Let $\phi = (\phi_x)_{x \in U_y}$ be an isomorphism of $\mathcal{U}|_{U_y}$ onto the trivial continuous field of Banach algebras over $U_y$  where, for each $y \in \Omega$, $\phi_x$ is an isometric isomorphism of Banach algebras.
For an arbitrary element $a_y \in A_y$, define a field $a$ to be equal to
$a(x) = f_y(x) (\phi^{-1}_x \circ \phi_y)(a_y)$, $x \in \Omega$. 
By Property (iv) of Definition \ref{cont-fields} and \cite[Proposition 10.1.9]{Di}, the field $a$ is continuous and $a \in \Theta$. Since $\|a(x)\| \to 0$ as $x \to \infty$, we have  $a \in \A$.
 \end{proof}

\begin{proposition}\label{A_t-unif-proj} Let $\Omega$ be a locally compact Hausdorff  space,  let $\mathcal{U} = \{\Omega,A_t,\Theta\}$
be a  locally trivial continuous field of Banach algebras,  and let the Banach algebra $\A$ be defined by $\mathcal{U}$. Suppose  $\A$ is projective in $\A$-$\mathrm{mod}$ ($\mathrm{mod}$-$\A$, $\A$-$\mathrm{mod}$-$\A$). Then the Banach algebras $A_{x}$, $x \in \Omega$, are uniformly left (right, bi-)projective.
\end{proposition}

\begin{proof} We shall prove the statement  for any   $\A$ which is  left projective.  Since $\A$ is left projective, there exists a morphism of left  Banach  $\A$-modules
 $\rho: \A \rightarrow \A_+ \widehat{\otimes} \A$ 
such that $\pi_{\A} \circ \rho = {\rm id}_{\A} $.

By assumption, $\mathcal{U}$ is locally trivial and so, for each $y \in \Omega$, there are open neighbourhoods $V_y$ and $U_y$ of $y$ such that $\overline{V_y} \subset U_y$, $\mathcal{U}|_{U_y}$ is trivial and $V_y$ is relatively compact. For each $y \in \Omega$, fix a continuous function $f_y\in C_{0}(\Omega)$ such that $0\leq f_y \leq 1$, $f_y(y)=1$ and $f_y|_{\Omega\setminus U_{y}}=0$. 
Let $\phi = (\phi_x)_{x \in U_y}$ be an isomorphism of $\mathcal{U}|_{U_y}$ onto the trivial continuous field of Banach algebras over $U_y$  where, for each $y \in \Omega$, $\phi_x$ is an isometric isomorphism of Banach algebras.
As in Lemma \ref{a-in-Theta}, for an arbitrary element $a_y \in A_y$, we define $a \in \A$ to be equal to
$a(x) = f_y(x) (\phi^{-1}_x \circ \phi_y)(a_y)$, $x \in \Omega$. 
We consider a map
\begin{align*}
\rho_{y}:A_{y} & \longrightarrow A_{y_{+}}\widehat{\otimes}A_{y}
\\
a_y & \longmapsto (\tau_{y_{+}}\otimes\tau_{y})\rho( f_y \phi^{-1}( \phi_y(a_y))).
\end{align*} 
It is easy to check that $\rho_{y}$ is a bounded linear operator and
\[ \|\rho_{y}\| = \sup_{\|a_y\|\leq 1}\|(\tau_{y_{+}}\otimes\tau_{y})\rho(f_y \phi^{-1} (\phi_y(a_y))\| \leq \|\tau_{y}\|^{2}\|\rho\|=\|\rho\|.
\]
We shall show that $\rho_{x}$ is a morphism of  Banach left $A_x$-modules. Since $\rho$ is a morphism of Banach left $\A$-modules, for all $a_x, b_x \in A_{x}$, we have
\begin{align*}
\rho_{x}(a_x b_x) & = (\tau_{x_{+}}\otimes\tau_{x})\rho(f_x \phi^{-1}\phi_{x}(a_x b_x))
\\
& = a_x (\tau_{x_{+}}\otimes\tau_{x})\rho(f_x^{\frac{1}{2}}\phi^{-1}\phi_{x}( b_x))
\\
& = 
(\tau_{x_{+}}\otimes\tau_{x})\rho(f_x \phi^{-1}\phi_{x}(a_x) f_x^{\frac{1}{2}}\phi^{-1}\phi_{x}(b_x))
\\
& = a_x (\tau_{x_{+}}\otimes\tau_{x})\rho(f_x \phi^{-1}\phi_{x}( b_x))
\\
& = a_x \rho_{x}(b_x).
\end{align*}
For all $a_x \in A_{x}$,
\begin{align*}
(\pi_{A_x}\circ\rho_{x})(a_x) & = \pi_{A_x}((\tau_{x_{+}}\otimes\tau_{x})\rho(f_x \phi^{-1} \phi_{x}(a_x)))
\\
& = \tau_{x}(\pi_{\A} \circ \rho)(f_x \phi^{-1} \phi_{x}(a_x))
\\
& = \tau_{x}(f_x\phi^{-1}\phi_{x}(a_x)) = a_x.
\end{align*}
Thus  $\rho_{x}$ is a morphism of Banach left $A_x$-modules
such that $\pi_{A_x}\circ\rho_{x}=\id_{\it A_{x}}$ and 
$ \|\rho_{x}\| \le \|\rho\|$. Therefore the Banach algebras $A_{x}$, $x \in \Omega$, are uniformly left projective.
\end{proof}
%%%%%%%%%%%%%%%%%%%%%%%%%%%%%%%

\begin{proposition}\label{dg-U-A_x-H^2} Let $\Omega$ be a locally compact Hausdorff  space, and let $\mathcal{U} = \{\Omega,A_t,\Theta\}$
 be either

{\rm (a)}  a continuous field of Banach algebras such that $\A$ defined by $\mathcal{U}$ has a right  bounded approximate identity; or 

{\rm (b)}  a  locally trivial continuous field of Banach algebras.\\ 
Suppose also that the Banach algebras $A_{x}$, $x \in \Omega$, are not uniformly left (right) projective. Then\\
{\rm (i)} there exists a Banach $\A$-bimodule $X$ such that 
$\H^2(\A,X) \neq \{0 \}$; and\\
{\rm (ii)} there exists a strongly unsplittable singular extension of the Banach algebra $\A$.
\end{proposition}
\begin{proof} By Proposition \ref{not-loc-triv-A_t-unif-proj} in case (a) and by Proposition \ref{A_t-unif-proj} in  case (b), $\A$ is not left  projective. By  \cite[Proposition IV.2.10(I)]{He4}, a Banach algebra $\A$ is left projective if and only if $ {\H}^2(\A,X) = \{0 \}$ for any right annihilator Banach  $\A$-bimodule $X$.
By \cite{Jo2} or \cite[Theorem I.1.10]{He4}, for a Banach $\A$-bimodule $X$, $ {\H}^2(\A,X) = \{0 \}$ if and only if all the singular extensions of $\A$ by  $X$ split strongly.
\end{proof}
%%%%%%%%%%%%%%%%%%%%%%%%%%%%%%%%%%%%

\begin{lemma}\label{A_x^2neq0} Let $\Omega$ be a locally compact Hausdorff  space,  let $\mathcal{U} = \{\Omega,A_t,\Theta\}$
 be a locally trivial continuous field of non-zero Banach algebras, and let the Banach algebra $\A$ be defined by $\mathcal{U}$. Suppose  $\A$ is projective in $\A$-$\mathrm{mod}$ or in $\mathrm{mod}$-$\A$. Then for each $x \in \Omega$, $A_{x}^2 \neq \{0\}$.
\end{lemma}
\begin{proof} It follows from Lemma \ref{A^2neq0} and Proposition \ref{A_t-unif-proj}. 
 \end{proof}

\begin{example} {\rm Let $\Omega$ be $\NN$ with the discrete topology. 
Consider  a continuous field of Banach algebras  $\mathcal{U}=\{{\NN}, A_{t}, \prod_{t \in \NN} A_t\}$ where 
$A_{t}$ is  the Banach algebra  $\ell^{2}_{t}$ of $t$-tuples of complex numbers $x =(x_{1},\ldots, x_{t})$ with pointwise multiplication and the norm $\|x\|=(\sum^{t}_{i=1}|x_{i}|^{2})^{\frac{1}{2}}$. Let $\A$ be defined by $\mathcal{U}$.

It is easy to see that, for each $t \in \NN$, the algebra $A_{t}$ is biprojective. For example, define
\begin{align*}
\tilde{\rho}_{t}:\ell^{2}_{t} & \longrightarrow \ell^{2}_{t}\widehat{\otimes} \ell^{2}_{t}
\\
(x_{1}\ldots x_{t}) & \longmapsto \sum_{k=1}^{t} x_{k} e^{k}\otimes e^{k},
\end{align*} 
where  $e^{k}= ( 0, \dots, 1, 0, \dots, 0) \in \ell^{2}_{t}$ with $1$ in the  $k$-th place. One can check that $\tilde{\rho}_{t}$ is a morphism of Banach 
$\ell^{2}_{t}$-bimodules such that $\pi_{A_t}\circ \tilde{\rho}_{t}=\id_{\ell^{2}_{t}}$.

Note that the algebra $A_{t}$ has the identity $e_{\ell^{2}_{t}} = ( 1, \dots, 1)$ and $\|e_{\ell^{2}_{t}} \|=\sqrt{t}$. Thus $\sup_{t \in \NN}\|e_{\ell^{2}_{t}} \| =\infty$.

{\it We will show that the  algebra $\A$ is not left projective.} In view of Proposition \ref{A_t-unif-proj} it is enough to show that the Banach algebras $A_{t}$, $t \in \NN$, are not uniformly left projective. 

We shall estimate $\| \rho_{t} \|$, $t \in \NN$, from below, where $\rho_{t}$ is  an arbitrary   morphism $\rho_{t}:\ell^{2}_{t} \longrightarrow \ell^{2}_{t}\widehat{\otimes} \ell^{2}_{t}$ of Banach left
$\ell^{2}_{t}$-modules which provides left projectivity of $\ell^{2}_{t}$.
For every element $e^{k} \in \ell^{2}_{t}$, 
\[ 
\rho_{t}(e^{k})=  e^{k} \rho_{t}(e^{k}) = e^{k} \otimes z^k
\]
for some $z^{k} \in \ell^{2}_{t}$ such that  $e^{k}  z^k = e^{k}$,
since $\pi_{\ell^{2}_{t}}\circ \rho_{t}=\id_{\ell^{2}_{t}}$.

We consider the bounded bilinear functional $\V : \ell^{2}_{t} \times \ell^{2}_{t} \longrightarrow  \CC $ defined by $\V(x, y) = \sum_{k=1}^{t} x_{k} y_{k}$, $x, y\in \ell^{2}_{t}$. Then, by the  universality property of the projective tensor product, the equation $V(x \otimes y) =\V(x, y)$ uniquely defines a continuous linear functional $V : \ell^{2}_{t} \widehat{\otimes} \ell^{2}_{t} \longrightarrow  \CC$ such that $\| V \|= \| \V \|$ =1.
Note that  the value of $V$ on the element 
\[
\rho_{t}\left(\sum_{k=1}^{t} \frac{1}{k} e^{k}\right) =\sum_{k=1}^{t} \frac{1}{k} e^{k} \otimes z^k
\]
is 
\[ V \left( \rho_{t}\left(\sum_{k=1}^{t} \frac{1}{k} e^{k}\right) \right) = \sum_{k=1}^{t} \frac{1}{k} > \log(t).\]
Thus, for each  $t \in \NN$, we have
\[
\log(t) < \left| V \left( \rho_{t}\left(\sum_{k=1}^{t} \frac{1}{k} e^{k} \right)\right) \right| \le \| V \|\| \rho_{t} \|\| \sum_{k=1}^{t} \frac{1}{k} e^{k}\|_{\ell^{2}_{t}}= \| \rho_{t} \| \left(\sum_{k=1}^{t} \frac{1}{k^2} \right)^{1/2}.
\]
Hence
\[
\| \rho_{t} \| > \frac{\log(t)}{ \left(\sum_{k=1}^{t} \frac{1}{k^2} \right)^{1/2}} \to \infty
\]
as $t \to \infty$. Therefore the Banach algebras $A_{t}$, $t \in \NN$, are not uniformly left projective and $\A$ is not left projective.

By Proposition \ref{dg-U-A_x-H^2},  there exists a Banach $\A$-bimodule $X$ such that $\H^2(\A,X) \neq \{0 \}$, and there exists a strongly unsplittable singular extension of the Banach algebra $\A$.
}
\end{example}

\section{Topological properties of $\Omega$ for  left projective Banach algebras of continuous fields over $\Omega$} \label{A-proj-Omega-paracompact}

In \cite[Theorem 4]{He3} Helemskii proved  that a closed ideal of a 
 commutative $C^*$-algebra $\A$
is projective in $\A$-$\mathrm{mod}$~ if and only if its spectrum is paracompact.
In this section we consider a $\sigma$-locally trivial  continuous field of (not necessarily commmutative) Banach algebras $\;\mathcal{U}=\{ \Omega,A_{x},\Theta \}$
and prove that the left projectivity of the Banach algebra $\A$ defined by $\mathcal{U}$ implies paracompactness of $\Omega$.

\begin{definition} A  Hausdorff topological space $\Omega$ is said to be {\em paracompact} if every open cover of $\Omega$ has an open locally finite refinement that is also a cover of $\Omega$. 
\end{definition}

\begin{proposition}\label{F(s,t)funct} Let $\Omega$ be a Hausdorff locally compact  space, let  
$\mathcal{U} = \{\Omega,A_t,\Theta\}$
 be a locally trivial continuous field of Banach algebras, and let the Banach algebra $\A$ be defined by $\mathcal{U}$. For each element  $v \in  \A \widehat{\otimes} \A$, the function 
\begin{align*}
F_v:\Omega \times\Omega &\longrightarrow \RR
\\
(s,t) & \longmapsto \|(\tau_{s}\otimes\tau_{t})(v)\|_{\it A_s\widehat{\otimes} A_t}
\end{align*}
satisfies the following conditions:
\begin{enumerate}
\item[{\rm (i)}] 
$F_v$ is  a positive continuous function on $\Omega \times\Omega$; 
\item[{\rm (ii)}] 
$F_v(s,t) \to 0$  as $ t \to \infty$ uniformly for $ s \in \Omega$,
\item[{\rm (iii)}] 
$F_v(s,t) \to 0$ as $ s \to \infty$ uniformly for $ t \in \Omega$,
\item[{\rm (iv)}]  If there is a morphism of left  Banach  $\A$-modules
 $\rho: \A \rightarrow \A_+ \widehat{\otimes} \A$ 
such that $\pi_{\A} \circ \rho = {\rm id}_{\A} $, then, for $a \in \overline{\A^2}$,
$F_{\rho(a)}(s,s) \ge \|a (s)\|$ for every $s \in \Omega$.
\end{enumerate}
\end{proposition}
\begin{proof} 
By \cite[Theorem 3.6.4]{Sh}, every element $v \in \A \widehat{\otimes} \A $ can be written as 
\[v= \sum_{i=1}^{\infty} \lambda_i~ a_i \otimes b_i,
\]
where  $\lambda_i \in {\mathbb C}, \; a_i , b_i \in \A,\;$ 
 $ \sum_{i=1}^{\infty} |\lambda_i| < \infty$ and the sequences
$\{a_i\}, \{b_i\}$ converge to zero  in $\A$ as $i \to \infty$ and so the sequences $\{a_i\}, \{b_i\}$ are bounded.
Therefore, we have
\begin{equation} \label{F_v=sum}
F_v(s,t)~=\|(\tau_{s}\otimes\tau_{t})(v)\|_{\it A_s\widehat{\otimes} A_t}
 = \| \sum_{i=1}^{\infty} \lambda_i~ {a_i}(s) 
\otimes b_i(t)\|_{A_s\widehat{\otimes} A_t}.
\end{equation}

(i) By assumption, $\mathcal{U}$ is locally trivial and so, for each $(s_0,t_0) \in \Omega \times\Omega$, there is an open neighbourhood $U_{s_0} \times V_{t_0}$ of $(s_0,t_0)$ such that $\mathcal{U}|_{U_{s_0}}$ and $\mathcal{U}|_{V_{t_0}}$ are trivial. 
 Let $\phi = (\phi_s)_{s \in U_{s_0}}$ and $\psi = (\psi_t)_{t \in V_{t_0}}$ be  isomorphisms of $\mathcal{U}|_{U_{s_0}}$  and $\mathcal{U}|_{V_{t_0}}$ respectively onto the trivial continuous fields of Banach algebras over $U_{s_0}$ and $V_{t_0}$  where $\phi_s$ and $\psi_t$ are isometric isomorphisms of Banach algebras $A_{s} \cong \tilde{A}_{s_0}$, $s \in U_{s_0}$, and $A_{t} \cong \tilde{A}_{t_0}$, $t \in V_{t_0}$, respectively.

Then, for every $\varepsilon >0$, there exists $N \in \NN$ such that 
$ \sum_{i=N+1}^{\infty} |\lambda_i|~ \|a_i\|  \|b_i\| < \varepsilon/4$ and there is an open neighbourhood $B_{s_0} \times D_{t_0} \subset U_{s_0} \times V_{t_0}$ of $(s_0,t_0)$ such that for all $(s,t) \in B_{s_0} \times D_{t_0}$
%%%%%%%%%%%%%%%%%%
\begin{align*}
~& \|(\phi_{s_0} \otimes \psi_{t_0} )(\tau_{s_0}\otimes \tau_{t_0})(v) -(\phi_{s} \otimes \psi_{t} )(\tau_{s}\otimes\tau_{t})(v) \|_{\tilde{A}_{s_0}\widehat{\otimes} \tilde{A}_{t_0}}
\\
\qquad & =  \|\sum_{i=1}^{\infty} \lambda_i~ \left(\phi(\bar{a_i})(s_0) \otimes \psi(\bar{b_i})(t_0) - \phi(\bar{a_i})(s) \otimes \psi(\bar{b_i})(t) \right)\|_{\tilde{A}_{s_0}\widehat{\otimes} \tilde{A}_{t_0}}
\end{align*}
\begin{align*}
~\qquad \qquad & \le 2 \sum_{i=N+1}^{\infty} |\lambda_i|~ \|a_i\|  \|b_i\| 
\\
& \;\;\;\;\;+
\|\sum_{i=1}^{N} \lambda_i~ \left(\phi(\bar{a_i})(s_0) \otimes \psi(\bar{b_i})(t_0) - \phi(\bar{a_i})(s_0) \otimes \psi(\bar{b_i})(t)\right)\|_{\tilde{A}_{s_0}\widehat{\otimes} \tilde{A}_{t_0}}
\\
& \;\;\;\;\;+
\|\sum_{i=1}^{N} \lambda_i~ \left(\phi(\bar{a_i})(s_0) \otimes \psi(\bar{b_i})(t) -\phi(\bar{a_i})(s) \otimes \psi(\bar{b_i})(t)\right)\|_{\tilde{A}_{s_0}\widehat{\otimes} \tilde{A}_{t_0}}
\end{align*}
\begin{align*}
 & \le 2 \sum_{i=N+1}^{\infty} |\lambda_i|~ \|a_i\|  \|b_i\| 
\\
& \;\;\;\;\;+
\sum_{i=1}^{N}  |\lambda_i|~ \|a_i\| 
\| \psi(\bar{b_i})(t_0) - \psi(\bar{b_i})(t)\|_{\tilde{A}_{t_0}}
\\
& \;\;\;\;\;+
\sum_{i=1}^{N} |\lambda_i|~  \|b_i\|  
\|\phi(\bar{a_i})(s_0) -\phi(\bar{a_i})(s) \|_{\tilde{A}_{s_0}} < \varepsilon,\qquad \qquad \quad 
\end{align*}
where $\bar{a_k} = a_k|_{U_{s_0}}$ and $\bar{b_k} = b_k|_{V_{t_0}}$ for $k=1,2,\dots$.
Hence
\[
\left| \|(\phi_{s_0} \otimes \psi_{t_0} )(\tau_{s_0}\otimes \tau_{t_0})(v)\|_{\tilde{A}_{s_0}\widehat{\otimes} \tilde{A}_{t_0}} -\|(\phi_{s} \otimes \psi_{t} )(\tau_{s}\otimes\tau_{t})(v) \|_{\tilde{A}_{s_0}\widehat{\otimes} \tilde{A}_{t_0}} \right|< \varepsilon.
\]
Since, for each $x \in \Omega$, $\phi_x$ and $\psi_x$ are isometric isomorphisms of Banach algebras we have
\[
\left| \|(\tau_{s_0}\otimes \tau_{t_0})(v)\|_{\it A_{s_0}\widehat{\otimes} A_{t_o}} -\|(\tau_{s}\otimes\tau_{t})(v) \|_{\it A_s\widehat{\otimes} A_t} \right|< \varepsilon.
\]
Therefore $F_v$ is a positive continuous function on $\Omega \times\Omega$.

(ii) and (iii).~ By (\ref{F_v=sum}), we have
\begin{equation} \label{F_v-to-0-1}
F_v(s,t)~= \| \sum_{i=1}^{\infty} \lambda_i~ {a_i}(s) \otimes b_i(t)\|_{\it A_s\widehat{\otimes} A_t}
~ \le \sum_{i=1}^{\infty} |\lambda_i|~ \|a_i(s)\|_{\it A_s}  \|b_i(t)\|_{\it A_t} .
\end{equation}
Recall that $ \sum_{i=1}^{\infty} |\lambda_i| < \infty$ and the sequences
$\{a_i\}, \{b_i\}$ converge to zero  in $A$ as $i \to \infty$ and so the sequences $\{a_i\}, \{b_i\}$ are bounded.
Hence for every  $\varepsilon >0$, there exists $N \in \NN$ such that 
\[ \sum_{i=N+1}^{\infty} |\lambda_i|~ \|a_i\|  \|b_i\| < \varepsilon/2
\]
 and a compact subset $K$ of $\Omega$ such that, for all $t \in \Omega \setminus K$,
\begin{equation} \label{F_v-to-0-2}
\sup_{s \in \Omega} F_v(s,t)~ = \sup_{s \in \Omega} \left(\| \sum_{i=1}^{N}  |\lambda_i|~ \|a_i\|  \|b_i(t)\|_{\it A_t} \right) +\varepsilon/2 < \varepsilon.
\end{equation}
Therefore $F_v(s,t) \to 0$  as $ t \to \infty$ uniformly for $ s \in \Omega$.
Similar one can show that
$F_v(s,t) \to 0$ as $ s \to \infty$ uniformly for $ t \in \Omega$.

(iv). Suppose there is a morphism of left  Banach  $\A$-modules
 $\rho: \A \rightarrow \A_+ \widehat{\otimes} \A$ 
such that $\pi_{\A} \circ \rho = {\rm id}_{\A} $. Then, for $a \in \overline{\A^2}$, $\rho (a) \in \A \widehat{\otimes} \A$. Note that, for every $s \in \Omega$,
\[\pi_{\it A_s}((\tau_{s}\otimes\tau_{s})\rho(a)) = a(s).\]
Thus, for every $s \in \Omega$,
\begin{align*}
  F_{\rho(a)}(s,s) ~&= \|(\tau_{s}\otimes \tau_{s})\rho(a)\|_{\it A_s\widehat{\otimes} A_s}  \\
~& \ge \| \pi_{\it A_s}((\tau_{s}\otimes\tau_{s})\rho(a))\| = \|a (s)\|.
\end{align*}
 \end{proof}

%%%%%%%%%%%%%%%%%%%%%%%
We will need the following notions.
Let $\Omega$ be a Hausdorff locally compact  space, and let 
$\mathcal{U} = \{\Omega,A_t,\Theta\}$ be a  continuous field of Banach algebras. For any condition $\Gamma$, we say that {\em $\mathcal{U}$  locally satisfies  condition $\Gamma$} if, for every $t \in \Omega$, there exists an open neighbourhood $V$ of $t$ such that $\mathcal{U}|_{V}$ satisfies  condition $\Gamma$.

By \cite[Theorem 5.17]{Ke},  a Hausdorff locally compact  space is regular. Therefore if $\mathcal{U}$ locally satisfies a condition $\Gamma$
on $\Omega$ then there is an open cover $\{ U_\mu \}$, $\mu \in \M$, of $\Omega$ such that each $\mathcal{U}|_{U_\mu}$ satisfies the condition $\Gamma$ and, in addition,
there is an open cover $\{ V_\nu\}$ of $\Omega$ such that $ \overline{V_\nu} \subset U_{\mu(\nu)}$  for each $\nu$ and some $\mu(\nu) \in \M$.

\begin{definition}\label{sigma-triv} We say that {\em $\mathcal{U}$ $\sigma$-locally ($n$-locally) satisfies a condition $\Gamma$} if there is an open cover $\{ U_\mu \}$, $\mu \in \M$, of $\Omega$ such that each $\mathcal{U}|_{U_\mu}$ satisfies the condition $\Gamma$ and, in addition,
there is a countable (cardinality $n$, respectively) open cover $\{ V_j\}$ of $\Omega$ such that $ \overline{V_j} \subset U_{\mu(j)}$  for each $j$ and some $\mu(j) \in \M$.
\end{definition}

\begin{remark} {\rm Let $\Omega$ be a Hausdorff locally compact space, and let $\mathcal{U} = \{\Omega,A_t,\Theta\}$  be a  continuous field of Banach algebras which  locally satisfies a condition $\Gamma$.
Suppose $\Omega$ is $\sigma$-compact (compact) and $\Omega_0$ is an open subset of $ \Omega$. Then $\mathcal{U}|_{U_{\Omega_0}}$ $\sigma$-locally ($n$-locally for some $n \in \NN$, respectively) satisfies the condition $\Gamma$.}
\end{remark}

\begin{definition} Let  $\Omega$ be a disjoint union of a family of open subsets $\{ W_\mu \}$, $\mu \in \M$, of $\Omega$. We
 say that {\em $\mathcal{U} = \{\Omega,A_t,\Theta\}$ is a disjoint union of $\mathcal{U}|_{W_{\mu}}$}, $\mu \in \M$.
\end{definition}

\begin{remark}\label{parac-disjoint-union} {\rm Let $\Omega$ be a paracompact Hausdorff locally compact space, and let 
$\mathcal{U} = \{\Omega,A_t,\Theta\}$
 be a  continuous field of Banach algebras which  locally satisfies a condition $\Gamma$. By \cite[Theorem 5.1.27 and Problem 3.8.C(b)]{En}, the space $\Omega$ is the disjoint union of open-closed $\sigma$-compacts $G_{\mu}$, $\mu \in \M$, of $\Omega$.
Suppose  $\Omega_0$ is an open subset of $ \Omega$. Then $\mathcal{U}|_{\Omega_0}$ is a disjoint union of $\mathcal{U}|_{G_{\mu}\cap \Omega_0}$, $\mu \in \M$,
$\sigma$-locally  satisfying the condition $\Gamma$. }
\end{remark}

\begin{theorem}\label{Omega-paracompact} Let $\Omega$ be a Hausdorff locally compact space, let $\mathcal{U} = \{\Omega,A_t,\Theta\}$ be a disjoint union of $\sigma$-locally  trivial continuous fields $\mathcal{U}|_{W_{\mu}}$, $\mu \in \M$,   of Banach algebras, and let the Banach algebra $\A$ be defined by $\mathcal{U}$. Suppose that $\A$ is left or right projective. Then  $\Omega$ is paracompact.
\end{theorem}
\begin{proof} By assumption, $\Omega $ is a  disjoint union of the family of open subsets $\{ W_\mu \}$, $\mu \in \M$, of $\Omega$. Let us prove that paracompactness of all $ W_\mu $, $\mu \in \M$, implies paracompactness of $\Omega$. Suppose, for every $\mu \in \M$, $ W_\mu $ is paracompact.  Let $\V=\{V_\alpha \}$
be an arbitrary open cover  of $\Omega$. For each $\mu \in \M$, the family
$V_\mu = \{ V \cap W_\mu: V \in \V \}$ is an open cover of $ W_\mu $. Since  $ W_\mu $ is paracompact, $V_\mu$ has an open locally finite refinement $\N_\mu$  that is also a cover of $ W_\mu $. Hence $\V$ has  an open locally finite 
refinement $\N = \cup_{\mu \in \M}\N_\mu$  of $\Omega$. Therefore  $\Omega$ is paracompact.

Let us choose $\mu \in \M$ and show that $ W_\mu $ is paracompact.
By assumption,  $\mathcal{U}|_{W_{\mu}}$ is a
$\sigma$-locally  trivial continuous field. By Definition \ref{sigma-triv}, there is an open cover $\{ U_\alpha \}$ of $W_\mu $ such that each $\mathcal{U}|_{U_\alpha}$ is locally trivial and, in addition,
there is a countable  open cover $\{ V_j\}$ of $W_\mu $ such that $ \overline{V_j} \subset U_{\alpha(j)}$  for each $j$.

\begin{lemma}\label{V_j-and-W_mu}
The paracompactness of all $ \overline{V_j}\cap W_\mu $, $j \in \NN$, implies paracompactness of $ W_\mu $.
\end{lemma}
\begin{proof}
 Let $\B$ be an arbitrary open cover  of $ W_\mu $. 
For each $j \in \NN$, the family
$\B_j = \{ B \cap \overline{V_j} \cap W_\mu: B \in \B \}$ is an open cover of $ \overline{V_j}\cap W_\mu $. By assumption,  $ \overline{V_j}\cap W_\mu $ is paracompact and so $\B_j$ has an open locally finite refinement $\D_j$  that is also a cover of  $ \overline{V_j}\cap W_\mu $.
The family of open subsets ~$\D'_j = \{ D \cap {V_j}: D \in \D_j\}$ is  locally finite in 
$ W_\mu $ and is a refinement  of $\B$. Furthermore, since $ W_\mu = \bigcup_{j \in \NN} V_j$, the family $\D = \bigcup_{j \in \NN}
\D'_j$ is an open $\sigma$-locally finite cover of $ W_\mu $.
By \cite[Theorem 5.28]{Ke},  $ W_\mu $ is paracompact.
 \end{proof}

Therefore to prove Theorem \ref{Omega-paracompact} it is enough to show that, for every  $\mu \in \M$ and $j \in \NN$, the topological space $ \overline{V_j} \cap W_\mu $ is paracompact. 

Since $\A$ is left projective, there exists a morphism of left  Banach  $\A$-modules
 $\rho: \A \rightarrow \A_+ \widehat{\otimes} \A$ 
such that $\pi_{\A } \circ \rho = {\rm id}_{\A} $.

Recall that $\mathcal{U}|_{U_\alpha(j)}$ is locally trivial and $ \overline{V_j} \subset U_{\alpha(j)}$. 
By Lemma \ref{A_x^2neq0}, there are continuous vector fields $x$ and $y$ on $ U_{\alpha(j)}$ such that $p(t)= x(t)y(t) \neq 0$
for every $t \in  U_{\alpha(j)}$. 

By \cite[Theorem 3.3.1]{En}, $\Omega$ is a Tychonoff space and so, for every $ s \in \overline{V_j} \subset U_{\alpha(j)}$, there is $f_s \in C_0(\Omega)$ such that $0 \le f_s \le 1$, $f_s(s)=1$ and $f_s(t) = 0$ for all $t \in \Omega \setminus U_{\alpha(j)}$. 
By Property (iv) of Definition \ref{cont-fields} and \cite[Proposition 10.1.9]{Di}, the field $f p$ is continuous and $\|f (t) p(t)\| \to 0$ as $t \to \infty$, we have  $f p \in \A$.
 For every $ s \in \overline{V_j} \subset U_{\alpha(j)}$ and $t \in \Omega$, we set
\[
\Phi(s,t) = F_{\rho(f_s p)}(s,t)
\]
where the function $F$ is defined in Proposition \ref{F(s,t)funct}. 

\begin{lemma}\label{Phi-well-defined}
The function $\Phi$ is well defined. 
\end{lemma}
\begin{proof}
Recall that $\rho$ is a morphism of left Banach $\A$-modules. Thus,
for every $ s \in \overline{V_j} \subset U_{\alpha(j)}$ and every $f_s, g_s \in C_0(\Omega)$ such that $0 \le f_s, g_s \le 1$, $f_s(s)= g_s(s)=1$ and $f_s(t) =g_s(t)= 0$ for all $t \in \Omega \setminus U_{\alpha(j)}$, we have, for $t \in \Omega$, 
\begin{align*}
  F_{\rho(f_s \; p)}(s,t) ~&= \|(\tau_{s}\otimes \tau_{t})\rho(f_s\; p)\|_{\it A_s\widehat{\otimes} A_t}  \\
 ~&= \|\sqrt{f_s(s)} x(s) (\tau_{s}\otimes \tau_{t})\rho(\sqrt{f_s} \;y)\|_{\it A_s\widehat{\otimes} A_t}\\
 ~&= \| (\tau_{s}\otimes \tau_{t})\rho(\sqrt{g_s} \sqrt{f_s} \;x y)\|_{\it A_s\widehat{\otimes} A_t}\\
 ~&= \|\sqrt{f_s(s)} x(s)  (\tau_{s}\otimes \tau_{t})\rho(\sqrt{g_s}\;y)\|_{\it A_s\widehat{\otimes} A_t}
\\
~&= \| (\tau_{s}\otimes \tau_{t})\rho(\sqrt{g_s} \; p)\|_{\it A_s\widehat{\otimes} A_t}\\
 ~&= \| (\tau_{s}\otimes \tau_{t})\rho(g_s\; p)\|_{\it A_s\widehat{\otimes} A_t}=  F_{\rho(g_s  p)}(s,t).
\end{align*}
Therefore $\Phi$ does not depend on the choice of $f_s$.
 \end{proof}

\begin{lemma}\label{Phi-continuous}
The function $\Phi$ is a continuous function on  $ \left(\overline{V_j}  \cap W_\mu \right) \times \Omega $. 
\end{lemma} 
\begin{proof}
Let $(s_0, t_0) \in 
\left(\overline{V_j}  \cap W_\mu\right) \times \Omega $
 and  $f_{s_0} \in C_0(\Omega)$ such that $0 \le f_{s_0} \le 1$, $f_{s_0}(s_0)=1$ and $f_{s_0}(t)= 0$ for all $t \in \Omega \setminus U_{\alpha(j)}$. Consider 
 a neighbourhood 
$V = U \times \Omega$ of the point  $(s_0, t_0)$ where
$ U = \{ s \in \overline{V_j}  \cap W_\mu : f_{s_0}(s) \neq 0 \}$.
Then, for $(s,t) \in V$,
\begin{align*}
 \Phi(s,t)~&= F_{\rho\left( \frac{f_{s_0}}{f_{s_0}(s)} \; p \right)}(s,t)\\
 ~&= \|(\tau_{s}\otimes \tau_{t})\rho \left(\frac{f_{s_0}}{f_{s_0}(s)}\; p \right)\|_{\it A_s\widehat{\otimes} A_t}  \\
 ~&= \frac{1}{f_{s_0}(s)}\| (\tau_{s}\otimes \tau_{t})\rho(f_{s_0} \;p)\|_{\it A_s\widehat{\otimes} A_t}\\
~&= \frac{1}{f_{s_0}(s)}\;F_{\rho(f_{s_0} \; p)}(s,t).
\end{align*}
Hence $\Phi$ is the ratio of two continuous functions on $V$, so it is continuous at $(s_0, t_0)$.
\end{proof}

\begin{lemma}\label{Phi-tends-to-0}
For every compact $K \subset \overline{V_j}  \cap W_\mu  $, the function $\Phi(s,t) \to 0$  as $ t \to \infty$ in $\Omega$ uniformly for $ s \in K$.
\end{lemma} 
\begin{proof}
By \cite[Theorem 3.1.7]{En}, since $\Omega $ is a Tychonoff space, for a compact subset $K \subset \left(\overline{V_j}  \cap W_\mu \right)\subset \Omega $ and for a closed subset $\Omega \setminus U_{\alpha(j)}\subset \Omega \setminus K$, there is $f_K \in C_0(\Omega)$ such that $0 \le f_K \le 1$, $f_K(s)=1$ for all $s \in K$ and $f_K(t) = 0$ for all $t \in \Omega \setminus U_{\alpha(j)}$. Note that $f_K p \in \A$.

By Proposition \ref{F(s,t)funct}, the function $F_{\rho\left( f_K \; p \right)}(s,t)\to 0$  as $ t \to \infty$ in $\Omega$ uniformly for $ s \in \Omega$.
Thus the function $\Phi(s,t) = F_{\rho(f_K p)}(s,t)$ on $K\times \Omega \subset \left(\overline{V_j}  \cap W_\mu \right) \times \Omega $
tends to $0$  as $ t \to \infty$ in $\Omega$ uniformly for $ s \in K$.
 \end{proof}

{\it Now let us complete the proof of Theorem} \ref{Omega-paracompact}.
For $(s,t) \in \left(\overline{V_j}  \cap W_\mu \right) \times \left(\overline{V_j}  \cap W_\mu \right)$, we set
\[
E(s,t) =\Phi(s,t)/\|p(s)\|.
\]
By Proposition \ref{F(s,t)funct}, $E(s,s) \ge 1$ for every  $ s \in \overline{V_j}  \cap W_\mu$. For $(s,t) \in \left(\overline{V_j}  \cap W_\mu \right) \times \left(\overline{V_j}  \cap W_\mu \right)$, we also set 
\[
G(s,t)= \min \{E(s,t);1 \} \min \{E(t,s);1 \}.
\]
By Lemmas \ref{Phi-well-defined}, \ref{Phi-continuous} and \ref{Phi-tends-to-0}, the function  $G(s,t)$ has the following properties: 
\begin{enumerate}
\item[{\rm (i)}] 
$G$ is  a positive continuous function on $\left(\overline{V_j}  \cap W_\mu \right) \times\left(\overline{V_j}  \cap W_\mu \right)$; 
\item[{\rm (ii)}] 
for any compact $K \subset \overline{V_j}  \cap W_\mu$,
$G(s,t) \to 0$  as $ t \to \infty$ uniformly for $ s \in K$ and
$G(s,t) \to 0$ as $ s \to \infty$ uniformly for $ t \in K$,
\item[{\rm (iii)}] $G(s,s)=1$. 
\end{enumerate}
By \cite[Theorem A.12, Appendix A]{He4}, $\overline{V_j}  \cap W_\mu$ is paracompact. By Lemma \ref{V_j-and-W_mu}, $\{ W_\mu \}$ is paracompact. Recall that $\Omega$  is a  disjoint union of the family of open subsets $\{ W_\mu \}$, $\mu \in \M$. Thus $\Omega$ is paracompact too.
 \end{proof}

%%%%%%%%%%%%%%%%%%%

\begin{theorem}\label{nonzero-field-Omega-paracompact} Let $\Omega$ be a Hausdorff locally compact  space, let 
$\mathcal{U} = \{\Omega,A_t,\Theta\}$
 be a locally  trivial continuous field  of Banach algebras, 
and let the Banach algebra $\A$ be defined by $\mathcal{U}$. Suppose that
 $\A$ is left or right projective and there is a continuous field $p \in \overline{\Theta^2}$ such that $p(t) \neq 0$ for all $t \in \Omega$ and $\sup_{s \in \Omega} \|p(s) \| < \infty$. Then  $\Omega$ is paracompact.
\end{theorem}
\begin{proof} We may assume that $\|p(s) \| =1$ for all $s \in \Omega$.
By \cite[Theorem 3.3.1]{En}, $\Omega$ is a Tychonoff space and so, for every $ s \in \Omega$, there is $f_s \in C_0(\Omega)$ such that $0 \le f_s \le 1$, $f_s(s)=1$. By \cite[Proposition 10.1.9 (ii)]{Di}, for every $f \in C_0(\Omega)$ such that $0 \le f_s \le 1$, the field $f p \in \overline{\Theta^2}$. Since $f(t) \to 0$ as $t \to \infty$, we have also $f p \in \overline{\A^2}$. For every $ s \in \Omega$ and $t \in \Omega$, we set
\[
\Phi(s,t) = F_{\rho(f_s p)}(s,t)
\]
where the function $F$ is defined in Proposition \ref{F(s,t)funct}.

\begin{lemma}\label{nonzero-field-Phi-well-defined}
The function $\Phi$ is well defined. 
\end{lemma}
\begin{proof}
Recall that $\rho$ is a morphism of left Banach $\A$-modules.
Since $\sqrt{f_s} p \in \overline{\A^2}$, it can be written as
\[
  \sqrt{f_s} p =\lim_\nu \sum_{k_\nu =1}^{n_\nu} x_{k_\nu} y_{k_\nu}.
\]
Thus,
for every $ s \in \Omega$ and every $f_s, g_s \in C_0(\Omega)$ such that $0 \le f_s, g_s \le 1$, $f_s(s)= g_s(s)=1$, we have, for $t \in \Omega$, 
\begin{align*}
  F_{\rho(f_s \; p)}(s,t) ~&= \|(\tau_{s}\otimes \tau_{t})\rho(f_s\; p)\|_{\it A_s\widehat{\otimes} A_t}  \\
 ~&= \|\lim_\nu \sum_{k_\nu =1}^{n_\nu}\sqrt{f_s(s)} x_{k_\nu}  (\tau_{s}\otimes \tau_{t})\rho(y_{k_\nu})\|_{\it A_s\widehat{\otimes} A_t}\\
~&= \|\lim_\nu \sum_{k_\nu =1}^{n_\nu} x_{k_\nu}  (\tau_{s}\otimes \tau_{t})\rho(y_{k_\nu})\|_{\it A_s\widehat{\otimes} A_t}\\
~&= \| (\tau_{s}\otimes \tau_{t})\rho(\sqrt{g_s}\sqrt{f_s} \; p)\|_{\it A_s\widehat{\otimes} A_t}\\
~&= \| (\tau_{s}\otimes \tau_{t})\rho(\sqrt{g_s} \; p)\|_{\it A_s\widehat{\otimes} A_t}\\
 ~&= \| (\tau_{s}\otimes \tau_{t})\rho(g_s\; p)\|_{\it A_s\widehat{\otimes} A_t}=  F_{\rho(g_s  p)}(s,t).
\end{align*}
Therefore $\Phi$ does not depend on the choice of $f_s$.
\end{proof}

\begin{lemma}\label{nonzero-field-Phi-continuous}
The function $\Phi$ is a continuous function on  $ \Omega \times \Omega $. 
\end{lemma} 
\begin{proof}
Let $(s_0, t_0) \in \Omega \times \Omega $
 and  $f_{s_0} \in C_0(\Omega)$ such that $0 \le f_{s_0} \le 1$ and $f_{s_0}(s_0)=1$.
Consider  a neighbourhood 
$V = U \times \Omega$ of the point  $(s_0, t_0)$ where
$ U = \{ s \in \Omega: f_{s_0}(s) \neq 0 \}$.
Then, as in Lemma \ref{Phi-continuous}, for $(s,t) \in V$,
\[
 \Phi(s,t)= \frac{1}{f_{s_0}(s)}\;F_{\rho(f_{s_0} \; p)}(s,t).
\]
Hence $\Phi$ is the ratio of two continuous functions on $V$, so it is continuous at $(s_0, t_0)$.
\end{proof}

\begin{lemma}\label{nonzero-field-Phi-tends-to-0}
For every compact $K \subset \Omega $, the function $\Phi(s,t) \to 0$  as $ t \to \infty$ in $\Omega$ uniformly for $ s \in K$.
\end{lemma} 
\begin{proof}
By \cite[Theorem 3.1.7]{En}, since $\Omega $ is a Tychonoff space, for a compact subset $K \subset  \Omega $, there is $f_K \in C_0(\Omega)$ such that $0 \le f_K \le 1$, $f_K(s)=1$ for all $s \in K$.  Note that $f_K p \in \A$.

By Proposition \ref{F(s,t)funct}, the function $F_{\rho\left( f_K \; p \right)}(s,t)\to 0$  as $ t \to \infty$ in $\Omega$ uniformly for $ s \in \Omega$.
Thus the function $\Phi(s,t) = F_{\rho(f_K p)}(s,t)$ on $K \times \Omega \subset \Omega \times \Omega $
tends to $0$  as $ t \to \infty$ in $\Omega$ uniformly for $ s \in K$.
\end{proof}

By Proposition \ref{F(s,t)funct}, $\Phi(s,s) \ge 1$ for every  $ s \in \Omega$. For $(s,t) \in \Omega \times \Omega$, we also set 
\[
G(s,t)= \min \{\Phi(s,t);1 \} \min \{\Phi(t,s);1 \}.
\]
By Lemmas \ref{nonzero-field-Phi-well-defined}, \ref{nonzero-field-Phi-continuous} and \ref{nonzero-field-Phi-tends-to-0}, the function  $G(s,t)$ has the following properties: 
\begin{enumerate}
\item[(i)] 
$G$ is  a positive continuous function on $\Omega \times \Omega$; 
\item[(ii)] 
for any compact $K \subset \Omega$,
$G(s,t) \to 0$  as $ t \to \infty$ uniformly for $ s \in K$ and
$G(s,t) \to 0$ as $ s \to \infty$ uniformly for $ t \in K$,
\item[(iii)] $G(s,s)=1$. 
\end{enumerate}
By \cite[Theorem A.12, Appendix A]{He4}, $\Omega$ is paracompact.
\end{proof}

\begin{proposition}\label{dg-U-Omega-H^2} Let $\Omega$ be a locally compact Hausdorff  space, and let 
$\mathcal{U} = \{\Omega,A_t,\Theta\}$ be either

{\rm (a)} a disjoint union of $\sigma$-locally  trivial continuous fields $\mathcal{U}|_{W_{\mu}}$, $\mu \in \M$,   of Banach algebras; or 

{\rm (b)}  a locally  trivial continuous field  of Banach algebras such that there is a continuous field $p \in \overline{\Theta^2}$ with $p(t) \neq 0$ for all $t \in \Omega$ and $\sup_{s \in \Omega} \|p(s) \| < \infty$.\\
Suppose $\Omega$ is not paracompact. Then, for the Banach algebra $\A$ defined by  $\mathcal{U}$,\\
{\rm (i)} there exists a Banach $\A$-bimodule $X$ such that 
$\H^2(\A,X) \neq \{0 \}$; and\\
{\rm (ii)} there exists a strongly unsplittable singular extension of the Banach algebra $\A$.
\end{proposition}
\begin{proof} By Theorem \ref{Omega-paracompact} in  case (a) and by Theorem \ref{nonzero-field-Omega-paracompact} in  case (b), $\A$ is not left (right)  projective. By  \cite[Proposition IV.2.10(I)]{He4}, a Banach algebra $\A$ is left projective if and only if $ {\H}^2(\A,X) = \{0 \}$ for any right annihilator Banach  $\A$-bimodule $X$. 
By \cite{Jo2} or \cite[Theorem I.1.10]{He4}, for a Banach $\A$-bimodule $X$, $ {\H}^2(\A,X) = \{0 \}$ if and only if all the singular extensions of $\A$ by  $X$ split strongly.
\end{proof}

\begin{corollary}\label{Ax-e-Omega-paracompact} Let  $\Omega$ be a locally compact Hausdorff  space,
 let $\mathcal{U} = \{\Omega,A_t,\Theta\}$ 
 be a locally trivial continuous field of Banach algebras such that every $\A_{t}$ has an identity $e_{A_{t}}$, $t \in \Omega$, and $\sup_{t\in\Omega}\|e_{A_{t}}\|\leq C$ for some constant $C \ge 1$, and let the Banach algebra $\A$ be defined by $\mathcal{U}$.  Suppose $\A$ is left projective. Then $\Omega$ is paracompact.
\end{corollary}
\begin{proof} Consider a field $ p \in \prod_{t \in \Omega}$ such that $p(t) = e_{A_{t}}$. By assumption, $\mathcal{U} = \{\Omega,A_t,\Theta\}$ 
 is a locally trivial continuous field of Banach algebras.
Thus, for every  $t \in \Omega$, there exist a neighbourhood $U_t$ of $t$ and $p' \in \Theta$ such that $p(s) = p'(s)$ for each $s \in U_t$. By Property (iv) of Definition \ref{cont-fields}, the field $p$ is continuous and $p \in \Theta$. Since $\A$ is left projective, by Theorem \ref{nonzero-field-Omega-paracompact}, $\Omega$ is paracompact.
\end{proof}

For any index set $\Lambda$, we shall  mean by $N(\Lambda)$ the set of finite subsets of $\Lambda$ ordered by inclusion.

\begin{proposition}\label{Ax-e-A-projective} Let  $\Omega$ be a locally compact Hausdorff space, let 
$\mathcal{U} = \{\Omega,A_t,\Theta\}$ 
 be a locally trivial continuous field of Banach algebras such that every $\A_{t}$ has an identity $e_{A_{t}}$ and $\sup_{t \in\Omega}\|e_{A_{t}}\|\leq C$ for some constant $C$, and let the Banach algebra $\A$ be defined by $\mathcal{U}$. Suppose $\Omega$ is paracompact. Then $\A$ is left projective.
\end{proposition}
\begin{proof}
By \cite[Lemma 2.1]{He3}, for any paracompact locally compact space $\Omega$ there exists an open cover $\{U_\mu \}_{\mu \in \Lambda}$  by relatively compact sets such that each point in $\Omega$ has a neighbourhood which intersects no more than three sets in $\{U_\mu \}_{\mu \in \Lambda}$. By assumption, $\Omega$ is paracompact and so, by \cite[Theorem 5.1.5]{En}, $\Omega$ is normal.
By \cite[Problem 5.W]{Ke}, since $\{U_\mu \}_{\mu \in \Lambda}$ is a locally finite open cover of the normal space $\Omega$, it is possible to select a non-negative continuous function $g_{\mu}$ for each $U_{\mu}$ in $\U$ such that $g_{\mu}$ is $0$ outside $U_{\mu}$ and is everywhere less than or equal to one, and 
\[
\sum_{\mu \in \Lambda}  g_{\mu}(s) =1\;\; \text{for all } \;\; s \in \Omega.
\]
In Corollary \ref{Ax-e-Omega-paracompact} we showed that a field $ p \in \prod_{t \in \Omega} A_{t}$ such that $p(t) = e_{A_{t}}$ is continuous and $p \in \Theta$. Note that $\sup_{t \in \Omega} \| p(t) \| \le C$.
By \cite[Proposition 10.1.9 (ii)]{Di}, for every $ \mu$, $g_{\mu} p \in \Theta$. Since $g_{\mu}(t) \to 0$ as $t \to \infty$, we have
$g_{\mu} p \in \A$.

For $a \in \A$ and  $\lambda \in N(\Lambda)$, define 
\[
u_{a,\lambda} = \sum_{\mu \in \lambda}
a \sqrt{g_\mu} \otimes \sqrt{ g_\mu} p. 
\]
As in  \cite{He3} we shall show that, for any $a \in \A$, the net $u_{a,\lambda} $ converges in $ \A \widehat{\otimes}\A$.
Note that any compact $K \subset \Omega$ intersects only a finite number of sets in the locally finite covering $\{U_\mu \}_{\mu \in \Lambda}$ and, for any $a \in \A$, $\|a(t)\| \to 0$ as $t \to \infty$.
Let $\varepsilon >0$. There is a finite set $\lambda \in N(\Lambda)$
such that for $\mu \not\in \lambda $ we have $\left\| a \sqrt{g_{\mu_p}}\right\|_{\A} < \frac{\varepsilon}{18 C} $.

For  $\lambda \preccurlyeq \lambda', \lambda'' $ we have
\begin{align*}
 \| u_{a,\lambda''} - u_{a,\lambda'}\|_{\A \widehat{\otimes} \A}&= 
 \| u_{a,\lambda'' \setminus \lambda} + u_{a,\lambda} - u_{a,\lambda' \setminus \lambda} -  u_{a,\lambda}\|_{\A \widehat{\otimes} \A}\\
& \le \| u_{a,\lambda'' \setminus \lambda}\|_{\A \widehat{\otimes} \A} + \| u_{a,\lambda' \setminus \lambda}\|_{\A \widehat{\otimes} \A}.
\end{align*}
By \cite[Theorem II.2.44]{He4}, for any $\tilde{\lambda} = \{\mu_1, \dots, \mu_m \}$,
\begin{align*}
\| u_{a,\tilde{\lambda}} \|_{\A \widehat{\otimes} \A}&= \left\|\sum_{\mu \in \tilde{\lambda}}
a \sqrt{g_\mu}  \otimes \sqrt{ g_\mu} p\right\|_{\A}\\
& \le \frac{1}{m} \sum_{l=1}^{m} \left\|\sum_{p=1}^{m} \xi^{l p} 
a \sqrt{g_{\mu_p}} \right\|_{\A} \times \left\|\sum_{q=1}^{m} \sum_{j=1}^{\tilde{n}} \xi^{-l q} \sqrt{ g_{\mu_q}} p \right\|_{\A}
\end{align*}
where $\xi$ is a primary $m$-th root of $1$  in $\CC$.

 For any $ s \in \Omega$,  there are no more than $3$ nonzero terms in the equality $\sum_{\mu \in \Lambda}  g_{\mu}(s) =1$. Therefore we have
\[
 \left\|\sum_{p=1}^{m} \xi^{l p} a \sqrt{g_{\mu_p}} \right\|_{\A} 
 = \sup_{s \in \Omega}\left\|\sum_{p=1}^{m} \xi^{l p} a(s) \sqrt{g_{\mu_p}(s)} \right\|_{\A} \le 3 \max_{1\le p\le m}\left\| a \sqrt{g_{\mu_p}} \right\|_{\A}.
\]
Hence     
\[
\| u_{a,\lambda'' \setminus \lambda}\|_{\A \widehat{\otimes} \A}  \le 9 \max_{\mu \not\in \lambda}\left\| a \sqrt{g_{\mu}}\right\|_{\A} \max_{\mu \not\in \lambda}\left\|\sqrt{g_{\mu}} p\right\|_{\A} \le \frac{\varepsilon}{2}.
\]
Let us define $\rho: \A \rightarrow \A \widehat{\otimes}\A$ by setting, for every $a \in \A$,
\[
\rho(a) = \lim_\lambda  u_{a,\lambda}.
\]
Thus, in view of the completeness of $\A \widehat{\otimes}\A$, the map $\rho$ is well defined and it is clear from the definition that  $\rho$ is  a morphism of left  Banach  $\A$-modules with $\| \rho \| \le 9C$. 
For every $a \in \A$,
\begin{align*}
(\pi_{\A} \circ \rho)(a) &=\lim_\lambda \sum_{\mu \in \lambda}
a  g_\mu p \\
&= \lim_\lambda \sum_{\mu \in \lambda}  a g_\mu = a.
\end{align*}
Thus $\pi_{\A} \circ \rho = {\rm id}_{\A} $.
Therefore   $\A$ is left projective. 
\end{proof}
%%%%%%%%%%%%%%%%%%%%

\begin{theorem}\label{Ax-with-e-H2} Let $\Omega$ be a Hausdorff locally compact space, 
let $\mathcal{U} = \{\Omega,A_t,\Theta\}$ 
 be a locally trivial continuous field of Banach algebras
 such that every $\A_{t}$ has an identity $e_{A_{t}}$, $t \in \Omega$, and $\sup_{t\in\Omega}\|e_{A_{t}}\|\leq C$
 for some constant $C$, and let the Banach algebra $\A$ be defined by $\mathcal{U}$. Then the following conditions  are equivalent:

{\rm (i)} $\Omega$ is paracompact;

{\rm (ii)} $\A$ is left projective;

{\rm (iii)} $\A$ is left projective and $\mathcal{U}$ is a disjoint union of $\sigma$-locally  trivial continuous fields of Banach algebras;

{\rm (iv)}
${\H}^2(\A,X) = \{0 \}$ for any right annihilator Banach  $\A$-bimodule $X$.
\end{theorem}
\begin{proof} By Proposition \ref{Ax-e-A-projective}, the fact that $\Omega$ is paracompact implies left projectivity of $\A$. Thus (i) $\Longrightarrow$ (ii). 
By Corollary \ref{Ax-e-Omega-paracompact}, the left projectivity of $\A$ implies that $\Omega$ is paracompact. Hence  (ii) $\Longrightarrow$ (i). 

By Remark \ref{parac-disjoint-union}, if
$\Omega$ is paracompact, then $\mathcal{U}$ is a disjoint union of $\sigma$-locally  trivial continuous fields of Banach algebras. By Theorem \ref{Omega-paracompact}, conditions (iii) implies  paracompactness of $\Omega$.
Thus (i) $\Longleftrightarrow$ (iii).

By \cite[Proposition IV.2.10(I)]{He4}, $\A$ is left projective if and only if $ {\H}^2(\A,X) = \{0 \}$ for any right annihilator Banach  $\A$-bimodule $X$ and so (ii)  $\Longleftrightarrow$ (iv). 
  \end{proof}

%%%%%%%%%%%%%%%%%%%%%%%%%%%

\section{Biprojectivity of  Banach algebras of continuous fields} \label{biprojective}

 Let $(A_{\lambda})_{\lambda \in \Lambda}$ be a family of $C^*$-algebras. 
Let $\A$ be the set of
$$ x = (x_{\lambda}) \in \prod_{\lambda \in \Lambda} A_{\lambda}$$
such that, for every $\varepsilon > 0$,  $~\|x_{\lambda}\| < \varepsilon $
except for finitely many  $\lambda$. Let 
$\|x \| = \sup_{\lambda \in \Lambda}\|x_{\lambda}\|$; then $\A$ with 
$\| \cdot \|$ is a $C^*$-algebra and is called {\it the direct sum}
or {\it the bounded product } of the $C^*$-algebras
$(A_{\lambda})_{\lambda \in \Lambda}$.

Recall Selivanov's result \cite{Se2} that any biprojective $C^*$-algebra is the direct sum of  $C^*$-algebras of the type $M_n(\CC)$. Another proof of this result is given in \cite[Theorem 5.4]{Ly02}. Therefore one can see that biprojective $C^*$-algebras can be described as  $C^*$-algebras $\A$ defined by a continuous field
$\mathcal{U}=\{\Lambda,A_{x},\prod_{x \in \Lambda} A_{x} \}$ where $\Lambda$ has the  discrete topology and the $C^*$-algebras $A_{x}$, $x \in \Lambda$, are of the type $M_n(\CC)$.

\begin{theorem}\label{Omega-discrete} Let $\Omega$ be a Hausdorff locally compact space, let $\mathcal{U}=\{\Omega,A_{x},\Theta \}$ be a locally  trivial continuous field  of Banach algebras, and let
the Banach algebra $\A$ be defined by $\mathcal{U}$. Suppose that $\A$ is biprojective. Then {\rm (i)} the Banach algebras $A_{x}$, $x \in \Omega$, are uniformly biprojective and {\rm (ii)} $\Omega$ is discrete.
\end{theorem}
\begin{proof} (i) By Proposition \ref{A_t-unif-proj}, the Banach algebras $A_{x}$, $x \in \Omega$, are uniformly biprojective. 

(ii) Since $\A$ is biprojective, there exists a morphism of  Banach  $\A$-bimodules
 $\rho: \A \rightarrow \A \widehat{\otimes} \A$ 
such that $\pi_{\A } \circ \rho = {\rm id}_{\A} $.

By \cite[Theorem 5.17]{Ke},  a Hausdorff locally compact space is regular. Therefore, since $\mathcal{U}$ is locally trivial
on $\Omega$,  there is an open cover $\{ U_\mu \}$, $\mu \in \M$, of $\Omega$ such that each $\mathcal{U}|_{U_\mu}$ is trivial and, in addition,
there is an open cover $\{ V_\alpha\}$ of $\Omega$ such that $ \overline{V_\alpha} \subset U_{\mu(\alpha)}$  for each $\nu$ and some $\mu(\alpha) \in \M$.

Let us show that, for every $\alpha$, ${V_\alpha}$ is discrete.
By Lemma \ref{A_x^2neq0}, there are continuous vector fields $x$ and $y$ on $ U_{\mu(\alpha)}$ such that $p(t)= x(t)y(t) \neq 0$
for every $t \in   U_{\mu(\alpha)}$.
By \cite[Theorem 3.3.1]{En}, $\Omega$ is a Tychonoff space and so, for every $ s \in {V_\alpha}$, there is $f_s \in C_0(\Omega)$ such that $0 \le f_s \le 1$, $f_s(s)=1$ and $f_s(t) = 0$ for all $t \in \Omega \setminus U_{\mu(\alpha)}$. Note that $f_s p \in \A$.
For every $ s \in {V_\alpha}$ and $t \in \Omega$, we set
\[
\Phi(s,t) = F_{\rho(f_s p)}(s,t)/\|p(s) \|,
\]
where the function $F$ is defined in Proposition \ref{F(s,t)funct}. 
By Proposition \ref{F(s,t)funct}, $\Phi(s,s) \ge 1$ for every 
$s \in {V_\alpha}$.

As in Lemmas \ref{Phi-well-defined} and  \ref{Phi-continuous}, the function  $\Phi(s,t)$ is a positive continuous function on ${V_\alpha} \times \Omega$ and does not depend on the choice of $f_s$.

Further, for every $ s,t \in {V_\alpha}$ such that $s \neq t$, there is $g \in C_0(\Omega)$ such that $0 \le g \le 1$, $g(s)=1$ and $g(t) = 0$. 
Since $\rho$ is a morphism of Banach $\A$-bimodules, we have 
\begin{align*}
\Phi(s,t) ~&=  F_{\rho(g\;f_s \; p)}(s,t)/\|p(s) \|\\
  ~&= \| (\tau_{s}\otimes \tau_{t})\rho(f_s \;x g\; y)\|_{\it A_s\widehat{\otimes} A_t}/\|p(s) \|\\
 ~&= \| (\tau_{s}\otimes \tau_{t})\rho(f_s \;x \;g\; y)\|_{\it A_s\widehat{\otimes} A_t}/\|p(s) \|\\
~&= \| (\tau_{s}\otimes \tau_{t})\rho(f_s \;x ) g(t)\; y(t)\|_{\it A_s\widehat{\otimes} A_t}/\|p(s) \|=0.
\end{align*}
Therefore, $\Phi(s,t)= 0$ for every $ s,t \in {V_\alpha}$ such that $s \neq t$, and $\Phi(s,s) \ge 1$ for every $s \in {V_\alpha}$.
Because $\Phi(s,t)$ is a positive continuous function on ${V_\alpha} \times {V_\alpha}$, this implies that  ${V_\alpha}$ is discrete.

Recall that $\Omega = \bigcup_\alpha {V_\alpha}$ where, for each $\alpha$, ${V_\alpha}$ is 
an open subset of $\Omega$. Thus $\Omega$ is discrete too.
  \end{proof}

A  Banach algebra $\A$ is said to be 
{\it contractible} if $\A_+$ is projective in
the category of  $\A$-bimodules.
 A  Banach  algebra $\A$ is
contractible if and only if $\A$ is biprojective and has an identity
 \cite[Def. IV.5.8]{He4}.

\begin{theorem}\label{A-contractible} Let $\Omega$ be a Hausdorff locally compact  space, let $\mathcal{U}=\{\Omega, A_{x},\Theta \}$ be a locally  trivial continuous field  of Banach algebras, and let the Banach algebra $\A$ be defined by $\mathcal{U}$.
Then the following conditions  are equivalent:

{\rm (i)} $\A$ is contractible;

{\rm (ii)} $\Omega$ is finite and discrete, and the Banach algebras $A_{x}$, $x \in \Omega$, are contractible;

{\rm (iii)} ${\H}^n(\A,X) = \{0 \}$ for any Banach  $\A$-bimodule $X$ and all $n \ge 1$.
\end{theorem}
\begin{proof} A finite direct sum of contractible Banach algebras is contractible and so (ii) $\Longrightarrow$ (i). 
By Theorem \ref{Omega-discrete}, if $\A$ is contractible then the Banach algebras $A_{x}$, $x \in \Omega$, are uniformly biprojective and  $\Omega$ is discrete. Since $\A$ has an identity $e$ and so $\|e(t)\| \ge 1$, for all $ t \in \Omega$, we have that $\Omega$ is compact.
Thus (i) $\Longrightarrow$ (ii).

By \cite[Proposition IV.5.8]{He4}, $\A$ is contractible if and only if 
${\H}^n(\A,X) = \{0 \}$ for all Banach  $\A$-bimodule $X$ and all $n \ge 1$, and so (i)  $\Longleftrightarrow$ (iii). 
  \end{proof}

%%%%%%%%%%%%%%%%%%%%%%%%%%%%%%%
In \cite[Examples 4.4 and 4.5]{Se00} there are descriptions of some biprojective Banach algebras $\A$ with very simple morphisms of $\A$-bimodules $\rho: \A \to \A \widehat{\otimes}\A$. We will use these algebras to constract examples of biprojective Banach algebras defined by continuous fields of Banach algebras over a discrete topological space.

\begin{example} {\rm Let $\Omega$ be a topological space with the  discrete topology. For every $t \in \Omega$,
let $E_t$ be an arbitrary Banach space of dimension $\dim E_t > 1$. Take a continuous linear functional $f_t \in E_t^*$, $\|f_t\|=1$ and define on 
$E_t$ the structure of a Banach algebra $A_{f_{t}}(E_{t})$ with multiplication given by $a b = f_t(a) b$, $a,b \in A_{f_{t}}(E_{t})$.
Choose $e_t \in E_t$ such that $f_t(e_t) =1$ and $\|e_t\|\le 2$. Then 
$e_t$ is a left identity of $A_{f_{t}}(E_{t})$. Consider the operator
$\rho_t:A_{f_{t}}(E_{t}) \to A_{f_{t}}(E_{t}) \widehat{\otimes} A_{f_{t}}(E_{t})$ defined $ a \mapsto e_t \otimes a$. It is easy to check that 
$\rho_t$ is an $A_{f_{t}}(E_{t})$-bimodule morphism, $\|\rho_t\| \le \|e_t\|\le 2$ and $\pi_{A_{f_{t}}(E_{t})}\circ \rho_t = {\rm id}_
{A_{f_{t}}(E_{t})} $, and so $A_{f_{t}}(E_{t})$ is a biprojective Banach algebra.

 Consider  the continuous field of Banach algebras  
$\mathcal{U}=\{\Omega, A_{t},\Theta \}$
 where $A_{t}$ is  the Banach algebra  $A_{f_{t}}(E_{t})$ with $f_t \in E_t^*$, $\|f_t\|=1$. Let $g_t$ be a continuous function on $\Omega$ such that $g_t(t)=1$
and $g_t(s)=0$ for all $s \neq t$. Since $\Omega$ has the discrete topology, by Example \ref{discrete-cont-fields},  $\Theta = \prod_{t \in \Omega}A_{t}$, and  so the field $e_t g_t $ such that $(e_t g_t)(s)=0$ for all $s \neq t$ and 
$(e_t g_t) (t)= e_t$ belongs to the Banach algebra $\A$ defined by 
$\mathcal{U}$.
Let $N(\Omega)$ be the set of finite subsets of $\Omega$ ordered by inclusion. For $a \in \A$ and  $\lambda \in N(\Omega)$, define 
\[
u_{a,\lambda} = \sum_{t \in \lambda} e_t g_t \otimes  g_t a. 
\]
By assumption,  $\sup_{t \in \Omega}\|e_t\| \le 2$. 
As in Proposition \ref{Ax-e-A-projective}, one can show that, for any $a \in \A$, the net $u_{a,\lambda}$ converges in $ \A \widehat{\otimes}\A$.
Let us define $\rho: \A \rightarrow \A \widehat{\otimes}\A$ by setting, for every $a \in \A$,
\[
\rho(a) = \lim_{\lambda}  \sum_{t \in \lambda} g_t e_t \otimes  g_t a.
\]
It is easy to check that   $\rho$ is  a morphism of   Banach  $\A$-bimodules and $\|\rho\| \le 2$. For every $a \in \A$,
\[
(\pi_{\A} \circ \rho)(a) =\lim_\lambda \sum_{t \in \lambda}
e_t  g_t a = \lim_\lambda \sum_{t \in \lambda}  a g_t = a.
\]
Thus $\pi_{\A} \circ \rho = {\rm id}_{\A} $.
Therefore   $\A$ is biprojective.

By \cite[Theorem V.2.28]{He4}, since $\A$ is biprojective,  ${\H}^n(\A,X) = \{0 \}$ for all Banach  $\A$-bimodule $X$ and all $n \ge 3$.
By \cite[Theorem 6]{Se95}, since $\A$ is biprojective and has a left bounded approximate identity,  ${\H}^n(\A,X^*) = \{0 \}$ for all dual Banach  $\A$-bimodule $X^*$ and all $n \ge 2$.
}
  \end{example}

\begin{example} {\rm Let $\Omega$ be a topological space with the discrete topology. For every $t \in \Omega$,
let $(E_t, F_t, \langle \cdot,\cdot \rangle_t)$ be  a pair of Banach spaces  with a non-degenerate continuous bilinear form  $\langle x,y \rangle_t$, $x \in E_t$, $y \in F_t$, with $\|\langle \cdot,\cdot \rangle_t\| \le 1$. The {\it tensor algebra $E_t \widehat{\otimes} F_t$ generated by the duality $\langle \cdot,\cdot \rangle_t$} can be constracted on the Banach space
$E_t \widehat{\otimes} F_t$ where the multiplication is defined by the formula
\[
(x_1 \otimes y_1) (x_2 \otimes y_2) =\langle x_2,y_1 \rangle_t x_1 \otimes y_2, \;\; x_i \in E_t, y_i \in F_t.
\]
Choose $x^0_t \in E_t, y^0_t \in F_t$ such that $\langle x^0_t,y^0_t \rangle_t =1$, $\|y^0_t\|=1$ and $\|x^0_t\|\le 2$.

 Consider the operator
$\rho_t:E_t \widehat{\otimes} F_t \to (E_t \widehat{\otimes} F_t) \widehat{\otimes} (E_t \widehat{\otimes} F_t)$ defined $ x \otimes y \mapsto (x \otimes y^0_t) \otimes (x^0_t \otimes y)$, $x \in E_t, y \in F_t$.
It is easy to check that $\rho_t$ is an $E_t \widehat{\otimes} F_t$-bimodule morphism, $\|\rho_t\| \le 2$ and $\pi_{E_t \widehat{\otimes} F_t}\circ \rho_t = {\rm id}_
{E_t \widehat{\otimes} F_t} $, and so $E_t \widehat{\otimes} F_t$ is a biprojective Banach algebra.

 Consider  the continuous field of Banach algebras 
$\mathcal{U}=\{\Omega, A_{t},\Theta \}$
  where $A_{t}$ is  the Banach algebra 
$E_t \widehat{\otimes} F_t$ with $\|\langle \cdot,\cdot \rangle_t\| \le 1$.
Let $\A$ be the Banach algebra defined by $\mathcal{U}$.
 Let $g_t$ be a continuous function on $\Omega$ such that $g_t(t)=1$
and $g_t(s)=0$ for all $s \neq t$.

 Since $\Omega$ has the discrete topology, by Example \ref{discrete-cont-fields},  $\Theta = \prod_{t \in \Omega}A_{t}$, and  so, for every $t \in \Omega$ and every $x_t\otimes y_t \in E_t \widehat{\otimes} F_t$, 
the field $ g_t x_t\otimes y_t$,  such that $(g_t x_t\otimes y_t)(s)=0$ for all $s \neq t$ and $(g_t x_t\otimes y_t) (t)= x_t\otimes y_t$ belongs to $\A$. Let $N(\Omega)$ be the set of finite subsets of $\Omega$ ordered by inclusion. For every $t \in \Omega$, choose $x^0_t \in E_t, y^0_t \in F_t$ such that $\langle x^0_t,y^0_t \rangle_t =1$, $\|y^0_t\|=1$ and $\|x^0_t\|\le 2$.
For $a = \{ x(t) \otimes y(t) \}_{t \in \Omega} \in \A$,  and  $\lambda \in N(\Omega)$, define 
\[
u_{a,\lambda} = \sum_{t \in \lambda}  (g_t x(t) \otimes y^0_t) \otimes (g_t x^0_t \otimes y(t)).
\] 
As in Proposition \ref{Ax-e-A-projective}, one can show that, for any $a \in \A$, the net $u_{a,\lambda}$ converges in $ \A \widehat{\otimes}\A$.
Let us define $\rho: \A \rightarrow \A \widehat{\otimes}\A$ by setting, for every $a \in \A$,
\[
\rho(a) = \lim_{\lambda}  \sum_{t \in \lambda} (g_t  x(t) \otimes y^0_t) \otimes (g_t x^0_t \otimes y(t)).
\]
It is easy to check that   $\rho$ is  a morphism of   Banach  $\A$-bimodules and $\|\rho\| \le 2$. For every $a \in \A$,
\[
(\pi_{\A} \circ \rho)(a) =\lim_\lambda \sum_{t \in \lambda}
g_t  (x(t) \otimes y^0_t)(x^0_t \otimes y(t))
= \lim_\lambda \sum_{t \in \lambda} g_t\; \langle x^0_t,y^0_t \rangle_t \;(x(t) \otimes y(t))   = a.
\]
Thus $\pi_{\A} \circ \rho = {\rm id}_{\A} $.
Therefore   $\A$ is biprojective.

By \cite[Theorem V.2.28]{He4}, since $\A$ is biprojective,  ${\H}^n(\A,X) = \{0 \}$ for all Banach  $\A$-bimodule $X$ and all $n \ge 3$.
}
  \end{example}

%%%%%%%%%%%%%%%%%%%%%%%%%%%%
%%%%%%%%%%%%%%%%%%%%%%%%%%%

\section{Left projectivity of $C^*$-algebras of some continuous fields } \label{left-proj-C*-posit-elem}

In \cite{Ly2} we proved that every separable $C^*$-algebra $\A$ and every closed left ideal of $\A$ are left projective. 
It is well-known that if $\A$ is separable, then $\A$ has a strictly positive element \cite{AaKa}. Indeed, if $\{y_n\}$ is dense in $\A$, then $\{x_n = y^*_n y_n\}$    is dense in $\A^+= \{x \in \A: x \ge 0\}$. Clearly $ x= \sum_{n =1}^{\infty}\frac{x_n}{2^n ||x_n||}$ is strictly positive
in $\A$. By \cite[Theorem 1]{AaKa}, a $C^*$-algebra 
$\A$ contains  a strictly positive element if and only if  
$\A$ has a countable increasing abelian approximate identity $e_i$, $i =1,2,\dots$, bounded by one. By \cite[Theorem 2.5]{PhR} or \cite[Theorem 1]{Ly2}, a $C^*$-algebra $\A$ with a countable increasing abelian approximate identity  is left and right projective. For  $C^*$-algebras $\A$, the relations between the separability, the existense of a strictly positive element in $\A$ and the left projectivity of $\A$ can be
summarised thus:
\[
\{\A \; \text{is separable} \} \subsetneq \{\A \; \text{has  a strictly positive element } \} \subsetneq \{\A \; \text{is left projective} \}. 
\]

\noindent For commutative $C^*$-algebras $\A$, that is, $\A =C_0(\widehat{\A})$ where the spectrum $\widehat{\A}$ is a Hausdorff locally compact  space, we have the following relations 

\begin{equation}\label{comm-A-sep-str_pos-proj}
\begin{array}{cccccc}
\A \; \text{is separable}     & ~ 
\begin{array}{cc} 
\Longrightarrow \\
\not\Longleftarrow
\end{array}                     & ~
\begin{array}{cc} 
\A \; \text{has a strictly} \\
\text{ positive element} 
\end{array}                      & ~ 
\begin{array}{cc} 
\Longrightarrow \\
\not\Longleftarrow
\end{array} & ~
\A \; \text{is left projective} \\
~ & ~ & ~ & ~ & ~\\
 \Updownarrow & ~ & ~ \Updownarrow & ~ & ~\Updownarrow \\
~ & ~ & ~ & ~ & ~\\
\begin{array}{cc} 
\widehat{\A} \; \text{is} \\
\text{ metrizable and has} \\
\text{a countable base} 
\end{array}   & ~
\begin{array}{cc} 
\Longrightarrow \\
\not\Longleftarrow
\end{array} & ~  
\widehat{\A} \; \text{is $\sigma$-compact}  &~
\begin{array}{cc}
\Longrightarrow \\
\not\Longleftarrow 
\end{array}  & ~ 
\widehat{\A} \; \text{is paracompact}.
\end{array}
\end{equation}
%%%%%%%%%%%%%%%%%

\begin{remark} {\rm (i) A Hausdorff locally compact  space $\Omega$ may be $\sigma$-compact without having a countable base for its topology, so $\A = C_0(\Omega)$ may have a strictly positive element without being separable. 

(ii) There are paracompact  spaces which are not $\sigma$-compact. For example,  any metrizable space is paracompact, but is not always $\sigma$-compact. The simple example  $\A=C_0(\RR)$ where $\RR$ is endowed with the discrete topology is a left projective $C^*$-algebra  without strictly positive elements.
}
\end{remark}
%%%%%%%%%%
We will prove in  Theorem \ref{count_Suslin_number} that,
for a commutative $C^*$-algebra $\A$ for which $\widehat{\A}$ has  countable Suslin number,   the condition of left projectivity is  equivalent to the existence of a strictly positive element, but not to the separability of $\A$; see Example \ref{nonsepA}.

In this section we also show that the class of left projective $C^*$-algebras includes noncommutative $C^*$-algebras defined by  some continuous fields (Theorem \ref{cont-field-strict-posit}).

The smallest cardinal number $m \ge {\aleph}_0$ such that every family 
of non-empty open, pairwise disjoint subsets of a topological space $\Omega$ has cardinality $\le m$ is called {\it the Suslin number} of $\Omega$ and it is denoted by $c(\Omega)$. The topological space $\Omega$ satisfies  {\it the Suslin condition} if $c(\Omega) ={\aleph}_0 $. 
A topological space $\Omega$ is called a {\it Lindel$\ddot{\rm o}$f space} if each open cover of $\Omega$ has a countable subcover.

\begin{theorem}\label{count_Suslin_number}
Let $\A$ be a commutative $C^*$-algebra,  so that $\A =C_0(\widehat{\A})$, and let $\widehat{\A}$ satisfy  the Suslin condition.  Then the following are equivalent:

 {\rm (i)} $\A$ is  projective in $\A$-$\mathrm{mod}$;

 {\rm (ii)} the spectrum $\widehat{\A}$ of $\A$ is  paracompact;

 {\rm (iii)} $\A$ contains a strictly positive element;

 {\rm (iv)} the spectrum $\widehat{\A}$ is $\sigma$-compact;

 {\rm (v)} $\widehat{\A}$ is a Lindel$\ddot{\it o}$f space;

{\rm (vi)} $\A$ has a sequential approximate identity bounded by one;

{\rm (vii)} ${\H}^2(\A,X) = \{0 \}$ for any right annihilator Banach  $\A$-bimodule $X$.
\end{theorem}
\begin{proof} If  the Suslin number of the topological space $\Omega$ is ${\aleph}_0 $ then, by \cite[Theorems 5.1.2, 5.1.25 and Exercise 3.8.C(b)]{En}, $\Omega$  is  paracompact if and only if $\Omega$ is $\sigma$-compact if and only if $\Omega$ is a Lindel$\ddot{\rm o}$f space.
Thus (ii) $\Longleftrightarrow$ (iv)  $\Longleftrightarrow$ (v).
 By \cite{AaKa}, $C_0(\Omega)$ contains a strictly positive element
if and only if $\Omega$ is $\sigma$-compact and so (iii) $\Longleftrightarrow$ (iv). By \cite[Theorem 4]{He3}, a 
 commutative $C^*$-algebra $\A$ is projective in $\A$-$\mathrm{mod}$~ if and only if its spectrum  $\widehat{\A}$ is paracompact. Hence (i) $\Longleftrightarrow$ (ii). By \cite[Proposition 12.7 and Lemma 12.9]{DW}, (iii) $\Longleftrightarrow$ (vi). By \cite[Proposition IV.2.10(I)]{He4},  $\A$ is left projective if and only if $ {\H}^2(\A,X) = \{0 \}$ for any right annihilator Banach  $\A$-bimodule $X$, and so (i) $\Longleftrightarrow$ (vii). 
  \end{proof}

\begin{lemma}\label{SuslinCond}~ {\rm \cite[Lemma 4.5]{Ly02}}
Let $\A$ be a commutative $C^*$-algebra  contained in ${\cal B}(H)$, where $H$ is a separable Hilbert space. Then the spectrum   $\widehat{\A}$ of  $\A$ satisfies the Suslin condition. 
\end{lemma}

\begin{corollary}\label{Hseparab_AinB(H)} Let $\A$ be a commutative
 $C^*$-algebra  contained in ${\cal B}(H)$, where $H$ is a separable 
Hilbert space. Then the conditions {\rm (i), (ii), (iii), (iv), (v), (vi) and (vii)} of Theorem {\rm \ref{count_Suslin_number}} are equivalent.
\end{corollary}
\begin{proof}
It follows from Theorem \ref{count_Suslin_number} and Lemma \ref{SuslinCond}.
  \end{proof}

\begin{lemma}\label{Omega-sep}~ {\rm \cite[Lemma 4.2]{Ly02}}
Let $\Omega$ be a separable Hausdorff locally compact space. Then the $C^*$-algebra $C_0(\Omega)$ is contained in ${\cal B}(H)$
for some separable Hilbert space $H$.
\end{lemma}

A Hausdorff topological space $X$ is {\it hereditarily paracompact} if each
 of its subspaces is  paracompact. By Stone's theorem, every metrizable 
topological space is paracompact \cite[5.1.3]{En}, and so is hereditarily paracompact.

\begin{example} \label{nonsepA} There exists a
nonseparable, hereditarily projective, commutative 
$C^*$-algebra $\A$ contained in ${\cal B}(H)$,
where $H$ is a separable Hilbert space.
\end{example}
\begin{proof} Let $ \Omega$ be a compact Hausdorff 
space. Recall that the commutative $C^*$-algebra
$C(\Omega)$ is separable if and only if $ \Omega$ is metrizable.
By \cite[Theorem 5]{He3} a commutative $C^*$-algebra
$C(\Omega)$ is hereditarily projective if and only if its spectrum 
$\Omega$ is hereditarily paracompact. 
Therefore, by Lemma \ref{Omega-sep}, it is enough to present 
a compact Hausdorff compact  space $ \Omega$ which is 
separable and hereditarily paracompact, but not metrizable. 
The topological space ``two arrows of Alexandrov" satisfies these
conditions. To describe the space,  consider two intervals
$X = [ 0, 1) $ and $X' = ( 0, 1] $ situated one above the other.
Let $\tilde{X} = X \cup X'$. A base for the topology of $\tilde{X}$ 
consists of all sets of the forms
$$ U = [ \alpha,\beta) \cup (\alpha',\beta')\;\; {\rm and}\;\;
V = (\alpha,\beta) \cup (\alpha',\beta'],$$
where $[ \alpha,\beta) \subset X$, while  $(\alpha',\beta')$
is the projection of $[ \alpha,\beta)$ on  $ X'$; and 
$(\alpha',\beta'] \subset X'$, while  $(\alpha,\beta)$
is the projection of $(\alpha',\beta']$ on  $ X$.
There is a description of properties of the  topological space 
$\tilde{X}$ in \cite[3.2.87]{AP}.
  \end{proof}

\begin{remark}\label{C_0(Omega,A)_strict_posit} {\rm Let $\Omega$ be $\sigma$-compact, and let $\A$  contain a strictly positive element. Then $C_0(\Omega, \A)$ contains a strictly positive element. Let  $h$ be a strictly positive element in $\A$ and $f$ be a strictly positive element in $C_0(\Omega)$. Then the element $a \in C_0(\Omega, \A)$ such that  $a(t) = f(t)h$, $t \in \Omega$,  is strictly positive. 
}
\end{remark}
%%%%%%%%%%%%
\begin{lemma}\label{bai-in-C0(Omega,A)}  Let $\Omega$ be a Hausdorff locally compact space, and let $\A$ be a Banach algebra with a bounded (say by $C \ge 1$) approximate identity $e_{n}$, $n=1,2, \dots.$ Then,  for every $a \in C_{0}(\Omega,\A)$, $\| a - a e_{n} \|_{C_{0}(\Omega,\A)} \to 0$ as $n \to \infty$.
\end{lemma}
\begin{proof} Let  $\varepsilon > 0$. Since $a \in C_{0}(\Omega,\A)$, 
there is a compact subset $K \subset \Omega$ such that $\| a (t) - a (t) e_{n} \| < \varepsilon$ for all $ t \in \Omega \setminus K$ and all $n$.

The element $a$ is continuous on $\Omega$ and so, for each $t \in K$,  
there is a neighbourhood $U_t$ of $t$ such that $\| a(s) - a (t) \| < \frac{\varepsilon}{3C}$ for all $s \in U_t$. The subset $K$ is compact, and  therefore one can find a finite family $U_{t_i}$, $ i =1, \dots, m$, such that $ K \subset \bigcup_{i=1}^m U_{t_i}$.

By assumption, $e_{n}$, $n=1,2, \dots$ is  a bounded  approximate identity of $\A$. Thus, for every $t_i$, $ i =1, \dots, m$, there is $N_i \in \NN$
such that $\| a(t_i) - a(t_i) e_{n} \|_{\A} < \varepsilon/3$ for all $n \ge N_i$. Take $N = \max \{N_1, \dots, N_m\}$. Then, for every $s \in K$, $s$ belongs to one of $U_{t_k}$. Hence, for all $n \ge N$, 
\begin{align*}
\| a(s) - a(s) e_{n} \|_{\A} & \le  \| a(s) - a(t_k) \|_{\A} +\| a(t_k) - a(t_k) e_{n} \|_{\A}\\
 ~& \;\;\;\;\;\; + \| (a(t_k) - a(s)) e_{n} \|_{\A} \\
~& < \frac{\varepsilon}{3C} + \frac{\varepsilon}{3} + C \frac{\varepsilon}{3C} < \varepsilon.
\end{align*}
Thus $\| a - a e_{n} \|_{C_{0}(\Omega,\A)} = \sup_{s \in \Omega}\| a(s) - a(s) e_{n} \|_{\A} \to 0$ as $n \to \infty$.
  \end{proof}

Recall that, by \cite[Theorem 1]{AaKa}, a $C^*$-algebra 
$\A$ contains  a strictly positive element if and only if  
$\A$ has a countable increasing abelian approximate identity bounded by one.  Such $C^*$-algebras are called {\it $\sigma$-unital}.
By \cite[Proposition 12.7]{DW} any $C^*$-algebra with a countable left approximate identity contains a strictly positive element. It follows 
that the following theorem is contained in \cite[Proof of Theorem 1]{Ly2}. We will need the inequalities from Theorem \ref{separable-case} below. 

\begin{theorem}\label{separable-case} Let $\A$ be a  $C^*$-algebra with a
countable left approximate identity. Then

{\rm (i)} there is a  countable increasing  approximate identity $u_n$, $n=0,\dots,$ ($u_0=0$) bounded by one such that, for any $a \in \A$, any sequence $(\eta_i)_{i \ge 1} \subset \CC$ with $|\eta_i| = 1$  and  any integers $n,m$ such that $m > n \ge 0$,
\begin{equation}\label{prop1}
\left\| \sum_{i=n+1}^{m} \eta_i
\sqrt{u_i -u_{i-1}} \right\|_{\A} \le 4,
\end{equation}
\begin{equation}\label{prop2}
\left\| a  \sum_{i=n+1}^{m} \eta_i
\sqrt{u_i -u_{i-1}} \right\|_{\A} \le 4 \| a \|_{\A}
\end{equation}
and
\begin{equation}\label{prop3}
\left\| a  \sum_{i=n+1}^{m} \eta_i
\sqrt{u_i -u_{i-1}} \right\|^2_{\A} \le 10\max_{n \le i \le m} \| a u_i - a \|_{\A}\| a \|_{\A} + \frac{1}{2^{2n-1}} \| a \|_{\A}^2;
\end{equation}

{\rm (ii)} the map 
\[
\rho_{\A}: \A \rightarrow \A \widehat{\otimes} \A: a \mapsto
\sum_{k=1}^{\infty}
a \;\sqrt{u_k -u_{k-1}}  \otimes  \; \sqrt{u_k -u_{k-1}} 
\]
is a morphism of left  Banach  $\A$-modules 
such that $\pi_{\A} \circ \rho = {\rm id}_{\A} $ and $\|\rho\|\le 16$, and therefore $\A$ is left  projective.

\end{theorem}

\begin{proposition}\label{C_0(omega,A-strict-posit)} Let $\Omega$ be a Hausdorff locally compact space, and let $\A$ be a C*-algebra which contains a strictly positive element. Suppose that  $\Omega$ is paracompact. Then   $C_{0}(\Omega,\A)$ is left projective.
\end{proposition}
\begin{proof}
By \cite[Lemma 2.1]{He3}, for any paracompact locally compact space $\Omega$ there exists an open cover $\U$ by relatively compact sets such that each point in $\Omega$ has a neighbourhood which intersects no more than three sets in $\U$.
By \cite[Problem 5.W]{Ke}, since $\U=\{U_\mu \}_{\mu \in \Lambda}$ is a locally finite open cover of the normal space $\Omega$, it is possible to select a non-negative continuous function $g_{\mu}$ for each $U_{\mu}$ in $\U$ such that $g_{\mu}$ is $0$ outside $U_{\mu}$ and is everywhere less than or equal to one, and 
\[
\sum_{\mu \in \Lambda}  g_{\mu}(s) =1\;\; \text{for all } \;\; s \in \Omega.
\]
By assumption $\A$ has a strictly positive element and so has
 a countable increasing  approximate identity bounded by one.
By Theorem \ref{separable-case}, there is  a  countable  increasing  approximate identity  $u_n$, $n=0,\dots,$ bounded by one in $\A$ which satisfies inequalities (\ref{prop1}), (\ref{prop2}) and (\ref{prop3}).

For $a \in C_0(\Omega, \A)$, $n \in \NN$ and  $\lambda \in N(\Lambda)$, define 
\[
u_{a,\lambda,n} = \sum_{\mu \in \lambda}\sum_{k=1}^{n}
a \sqrt{g_\mu}\;\sqrt{u_k -u_{k-1}}  \otimes \sqrt{ g_\mu} \; \sqrt{u_k -u_{k-1}} . 
\]
In Lemma \ref{u_a_lambda_n-converges} below we will prove in a more general case that, for any $a \in C_0(\Omega, \A)$, the net $u_{a,\lambda,n} $ converges in $ C_0(\Omega, \A) \widehat{\otimes}C_0(\Omega, \A)$.
%%%%%%%%%%%%%
Let us define $\rho: C_0(\Omega, \A) \rightarrow C_0(\Omega, \A) \widehat{\otimes}C_0(\Omega, \A)$ by setting, for every $a \in C_0(\Omega, \A)$,
\[
\rho(a) = \lim_\lambda \lim_{n \to \infty} u_{a,\lambda,n}.
\]
Thus, in view of the completeness of $C_0(\Omega, \A) \widehat{\otimes}C_0(\Omega, \A)$, the map $\rho$ is well defined and it is clear from the definition that  $\rho$ is  a morphism of left  Banach  $C_0(\Omega, \A)$-modules and $\|\rho\|\le 9 \times 16$. 
By Lemma \ref{bai-in-C0(Omega,A)}, for every $a \in C_0(\Omega, \A)$,
\begin{align*}
(\pi_{C_0(\Omega, \A)} \circ \rho)(a) &=\lim_\lambda \lim_{n \to \infty} \left(\sum_{\mu \in \lambda}\sum_{k=1}^{n}
a g_\mu\;\left(u_k -u_{k-1} \right) \right)\\
&= \lim_\lambda \sum_{\mu \in \lambda} \lim_{n \to \infty} a g_\mu\;u_n \\
&= \lim_\lambda \sum_{\mu \in \lambda}  a g_\mu = a.
\end{align*}
Thus $\pi_{C_0(\Omega, \A)} \circ \rho = {\rm id}_{C_0(\Omega, \A)} $.
Therefore   $C_{0}(\Omega,\A)$ is left projective.
  \end{proof}

Furthermore we will need  \cite[Theorem II.2.44]{He4} in a more general form.

\begin{lemma}\label{diagonal-element} Let $E, F$ be Banach spaces.  Suppose an element $u \in E \widehat{\otimes} F$ can be presented as 
\[ u = \sum_{l =1}^m \sum_{k=1}^{n} x^{\mu_l}_k \otimes y^{\mu_l}_k,
\]
$\xi$ is a primary $m$-th root of $1$ and $\eta$ is a primary $n$-th root of $1$ in $\CC$.
Then 
\[ 
\|u\| \le \frac{1}{m n} \sum_{l=1}^{m}\sum_{k=1}^{n} 
\left\|\sum_{t=1}^{m} \sum_{i=1}^{n} \xi^{l t} \eta^{k i} x^{\mu_t}_i \right\| \left\|\sum_{s=1}^{m} \sum_{j=1}^{n} \xi^{-l s} \eta^{-k j} y^{\mu_s}_j\right\|.
\]
\end{lemma}
\begin{proof} We consider the following element $v$ in $E \widehat{\otimes} F$ 
\[
v = \frac{1}{m n} \sum_{l=1}^{m}\sum_{k=1}^{n} 
\left(\sum_{t=1}^{m} \sum_{i=1}^{n} \xi^{l t} \eta^{k i} x^{\mu_t}_i \right)\otimes \left(\sum_{s=1}^{m} \sum_{j=1}^{n} \xi^{-l s} \eta^{-k j} y^{\mu_s}_j\right).
\]
By the definition of the norm in $E \widehat{\otimes} F$,
\[
\|v\| \le \frac{1}{m n} \sum_{l=1}^{m}\sum_{k=1}^{n} 
\left\|\sum_{t=1}^{m} \sum_{i=1}^{n} \xi^{l t} \eta^{k i} x^{\mu_t}_i \right\| \left\|\sum_{s=1}^{m} \sum_{j=1}^{n} \xi^{-l s} \eta^{-k j} y^{\mu_s}_j\right\|.
\]
Note that  $$\lambda_{t s} = \sum_{l=1}^{m} \xi^{l t} \xi^{-l s} =\sum_{l=1}^{m} \xi^{l( t- s)} = 
%\xi^{t-s}\frac{\xi^{m(t-s)}-1}{\xi^{t-s}-1}= 
m \delta^t_s$$ and
$$\gamma_{i j} = \sum_{k=1}^{n} \eta^{k i}\eta^{-k j} = \sum_{k=1}^{n} \eta^{k(i-j)} = 
%\eta^{i-j} \frac{\eta^{n(i-j)}-1}{\eta^{i-j}-1}=
 n \delta^i_j$$
where $\delta^t_s$ is the Kronecker symbol. Therefore, we have
\begin{align*}
 v&= \frac{1}{m n} \sum_{l=1}^{m} \sum_{k=1}^{n} 
\left(\sum_{t=1}^{m} \sum_{i=1}^{n} \xi^{l t} \eta^{k i} x^{\mu_t}_i \right)\otimes \left(\sum_{s=1}^{m} \sum_{j=1}^{n} \xi^{-l s} \eta^{-k j} y^{\mu_s}_j\right)\\
 ~&=  \sum_{t,s=1}^{m} \frac{1}{m}  \lambda_{t s}\;
\sum_{i,j =1}^{n}   \frac{1}{ n} \gamma_{i j}\;
x^{\mu_t}_i \otimes y^{\mu_s}_j \\
~& =\sum_{l=1}^{m} \sum_{k=1}^{n} x^{\mu_l}_k \otimes y^{\mu_l}_k =u.
\end{align*}
  \end{proof}

%%%%%%%%%%%%%%%%%%%%%%%%%%%%%%
\begin{theorem}\label{cont-field-strict-posit}  Let $\Omega$ be a Hausdorff locally compact space, let 
 $\mathcal{U}=\{\Omega, A_{x}, \Theta \}$ be a locally trivial continuous field of $C^*$-algebras where each $A_{x}$ contains a strictly positive element, let $\ell \in \NN$, and let the $C^*$-algebra $\A$ be defined by $\mathcal{U}$. Suppose $\Omega$ is paracompact and one of the following conditions is satisfied:\\
{\rm (i)} the topological dimension  $\dim \Omega$ of $\Omega$ is finite and no greater than  $\ell$, or\\
{\rm (ii)}~ $\mathcal{U}$ is $\ell$-locally trivial.\\
Then  $\A$ is left projective. 
\end{theorem}

\begin{proof} (i) By assumption, $\mathcal{U}$ is locally trivial and so, for each $s \in \Omega$, there is an open neighbourhood $U_{s}$ of $s$ such that $\mathcal{U}|_{U_{s}}$ is trivial. Since $\Omega$ is paracompact,
the open cover $\{U_{s} \}_{s \in \Omega} $ of $\Omega$ has an open locally finite refinement $\{W_{\nu} \}$ that is also a cover of $\Omega$.

By  \cite[Theorem 5.1.5]{En},  paracompactness of $\Omega$ implies that  $\Omega$ is a normal topological space. By \cite[Theorem 7.2.4]{En}, for the normal space $\Omega$, the topological dimension
$\dim \Omega \le \ell$ implies that the locally finite open cover $\{W_{\nu} \}$ of $\Omega$ possesses an open locally finite refinement $\{V_{\mu} \}_{\mu \in \Lambda}$ of order $\ell$ that is also a cover of $\Omega$.

Let $\phi^\mu = (\phi^\mu_x)_{x \in V_{\mu}}$  be an isomorphism of $\mathcal{U}|_{V_{\mu}}$  onto the trivial continuous field of $C^*$-algebras over $V_{\mu}$  where, for each $x \in \Omega$, $\phi^\mu_x$ is an  isomorphism of $C^*$-algebras $A_x \cong \tilde{A}_{\mu}$.

By assumption $\tilde{A}_{\mu}$ has a strictly positive element and so has
 a countable  increasing  approximate identity bounded by one.
By Theorem \ref{separable-case}, there is  a  countable  increasing  approximate identity  $u^{\mu}_n$, $n=0,\dots,$ in $\tilde{A}_{\mu}$ bounded by one which satisfies inequalities (\ref{prop1}), (\ref{prop2}) and (\ref{prop3}).

By \cite[Problem 5.W]{Ke}, since $\{V_{\mu} \}_{\mu \in \Lambda}$ is a locally finite open cover of the normal space $\Omega$, it is possible to select a non-negative continuous function $f_{\mu}$ for each $V_{\mu}$ in $\{V_{\mu} \}_{\mu \in \Lambda}$ such that $f_{\mu}$ is $0$ outside $V_{\mu}$ and is everywhere less than or equal to one, and 
\[
\sum_{\mu \in \Lambda}  f_{\mu}(s) =1\;\; \text{for all } \;\; s \in \Omega.
\]
Note that in the equality $\sum_{\mu \in \Lambda}  f_{\mu}(s) =1$, for any 
$ s \in \Omega$, there are no more than $\ell$ nonzero terms.

\begin{lemma}\label{u_a_lambda_n}  For any $a \in \A$ and  for any $\lambda = \{\mu_1, \dots,  \mu_N\}$, 
\begin{equation}\label{prop-1-general}
\left\|\sum_{p=1}^{N} \sum_{i=1}^{\tilde{n}} \xi^{l p} \eta^{k i}
a \sqrt{f_{\mu_p}}\;(\phi^{\mu_p})^{-1}\left(\sqrt{u_i^{\mu_p} -u_{i-1}^{\mu_p}}\right) \right\|_{\A} \le  4 \ell \max_{1\le p\le N}\left\| a \sqrt{f_{\mu_p}}\right\|_{\A}
\end{equation}
and
\begin{equation}\label{prop-2-general}
\left\|\sum_{q=1}^{N} \sum_{j=1}^{\tilde{n}} \xi^{-l q} \eta^{-k j} 
\sqrt{ f_{\mu_q}} \;(\phi^{\mu_q})^{-1}\left( \sqrt{u_j^{\mu_q} -u_{j-1}^{\mu_q}} \right)\right\|_{\A} \le 4 \ell
\end{equation}
where $\xi$ is a primary $N$-th root of $1$ and $\eta$ is a primary $\tilde{n}$-th root of $1$ in $\CC$,
and
\begin{align}\label{prop-3-general}
~& \left\|\sum_{p=1}^{N} \sum_{i=n + 1}^{m} \xi^{l p} \eta^{k i}
a \sqrt{f_{\mu_p}}\;(\phi^{\mu_p})^{-1}\left(\sqrt{u_i^{\mu_p} -u_{i-1}^{\mu_p}}\right) \right\|_{\A}  \\
 & \le \ell \max_{1\le p\le N}  \biggl(
10\max_{n \le i \le m} \left\|\phi^{\mu_p}( a\sqrt{f_{\mu_p}}) u_i^{\mu_p} - \phi^{\mu_p}(a \sqrt{f_{\mu_p}})\right\|_{C_0(V_{\mu_p},\tilde{A}_{\mu_p})} \;\left\|a\sqrt{f_{\mu_p}} \right\|_{\A}\\
~ & \quad 
 + \frac{1}{2^{2n-1}} \left\|a\sqrt{f_{\mu_p}} \right\|_{\A}^2 \biggr)^{1/2}
\end{align}
where $\xi$ is a primary $N$-th root of $1$ and $\eta$ is a primary $(m-n)$-th root of $1$ in $\CC$.
%%%%%%%%%%%%%%%%%%%%%
\end{lemma}
\begin{proof}  For any $ s \in \Omega$,  there are no more than $\ell$ nonzero terms in the equality $\sum_{\mu \in \Lambda}  f_{\mu}(s) =1$. By Theorem \ref{separable-case}, for every $\mu$,
\[
 \left\| \sum_{i=1}^{\tilde{n}} \eta^{k i}
\sqrt{u_i^{\mu} -u_{i-1}^{\mu}} \right\|_{\tilde{A_{\mu}}} \le 4.
\]
Hence we have
\begin{align*}
~ & \left\|\sum_{p=1}^{N} \sum_{i=1}^{\tilde{n}} \xi^{l p} \eta^{k i}
a \sqrt{f_{\mu_p}}\;(\phi^{\mu_p})^{-1}\left(\sqrt{u_i^{\mu_p} -u_{i-1}^{\mu_p}}\right) \right\|_{\A} \\
& = \sup_{s \in \Omega}\left\|\sum_{p=1}^{N}  \xi^{l p} a(s) \sqrt{f_{\mu_p}(s)}\; \sum_{i=1}^{\tilde{n}} \eta^{k i}
(\phi^{\mu_p}_s)^{-1}\left(\sqrt{u_i^{\mu_p} -u_{i-1}^{\mu_p}}\right) \right\|_{A_s} \\
& \le \ell \max_{1\le p\le N}\left\| a \sqrt{f_{\mu_p}}\; \sum_{i=1}^{\tilde{n}} \eta^{k i}
(\phi^{\mu_p})^{-1}\left(\sqrt{u_i^{\mu_p} -u_{i-1}^{\mu_p}}\right) \right\|_{\A} \\
& \le 4 \ell \max_{1\le p\le N}\left\| a \sqrt{f_{\mu_p}}\right\|_{\A}. 
\end{align*}
Thus the inequalities (\ref{prop-1-general}) and (\ref{prop-2-general})
hold.
 
%%%%%%%%%%%%%%%%%%
 By Theorem \ref{separable-case}, for every $\mu$,  every $\gamma_i \in \CC$ with $|\gamma_i| = 1$, $ i \in \NN$, and  every $b \in {\tilde{A_{\mu}}}$,
\[
\left\|\sum_{i=n+1}^{m} \gamma_i\; b \sqrt{u_i^{\mu} - u_{i-1}^{\mu}} 
\right\|^2_{\tilde{A_{\mu}}} \le 10\max_{n \le i \le m} \| b u_i^{\mu} - b \|_{\tilde{A}_{\mu}}\| b \|_{\tilde{A}_{\mu}} + \frac{1}{2^{2n-1}} \| b \|_{\tilde{A}_{\mu}}^2.
\]
Therefore we have 
\begin{align*}
~ &\left\|\sum_{p=1}^{N} \sum_{i=n + 1}^{m} \xi^{l p} \eta^{k i}
a \sqrt{f_{\mu_p}}\;(\phi^{\mu_p})^{-1}\left(\sqrt{u_i^{\mu_p} -u_{i-1}^{\mu_p}}\right) \right\|_{\A} \\
& = \sup_{s \in \Omega}\left\|\sum_{p=1}^{N} \sum_{i=n + 1}^{m} \xi^{l p} \eta^{k i}
a(s) \sqrt{f_{\mu_p}(s)}\;(\phi_s^{\mu_p})^{-1}\left(\sqrt{u_i^{\mu_p} -u_{i-1}^{\mu_p}}\right) \right\|_{\A}\\
& \le \ell \max_{1\le p\le N} \left\|\sum_{i=n + 1}^{m} \eta^{k i} a \sqrt{f_{\mu_p}}\;(\phi^{\mu_p})^{-1} \left(\sqrt{u_i^{\mu_p} -u_{i-1}^{\mu_p}}\right) \right\|_{\A}\\
& = \ell \max_{1\le p\le N} \sup_{s \in \Omega}\left\|\sum_{i=n + 1}^{m} \eta^{k i} a(s) \sqrt{f_{\mu_p}(s)}\;(\phi_s^{\mu_p})^{-1} \left(\sqrt{u_i^{\mu_p} -u_{i-1}^{\mu_p}}\right) \right\|_{A_s}\\
~ & \le \ell \max_{1\le p\le N}  \biggl(
10\max_{n \le i \le m} \left\|\phi^{\mu_p}( a\sqrt{f_{\mu_p}}) u_i^{\mu_p} - \phi^{\mu_p}(a \sqrt{f_{\mu_p}})\right\|_{C_0(V_{\mu_p},\tilde{A}_{\mu_p})} \;\left\| a\sqrt{f_{\mu_p}} \right\|_{\A}\\
~ & \quad  + \frac{1}{2^{2n-1}} \left\|a\sqrt{f_{\mu_p}} \right\|_{\A}^2 \biggr)^{1/2}.
\end{align*}
  \end{proof}

For $a \in \A$, $n \in \NN$ and  $\lambda \in N(\Lambda)$, define an element $u_{a,\lambda,n}$ in $\A$
\[
u_{a,\lambda,n} = \sum_{\mu \in \lambda}\sum_{k=1}^{n}
a \sqrt{f_\mu}\;\; (\phi^\mu)^{-1}\left(\sqrt{u_k^\mu -u_{k-1}^\mu} \right) \otimes \sqrt{ f_\mu} \;\;(\phi^\mu)^{-1}\left( \sqrt{u_k^\mu -u_{k-1}^\mu} \right). 
\]
Define $N(\Lambda) \times \NN$ as a directed set with $
(\lambda', n) \preccurlyeq (\lambda'', m) $ if and only if $ \lambda' \subset \lambda''$ and $n  \le m$.

\begin{lemma}\label{u_a_lambda_n-converges} For any $a \in \A$,
the net $u_{a,\lambda,n} $ converges in $\A \widehat{\otimes} \A$.
\end{lemma}
\begin{proof} Note that any compact $K \subset \Omega$ intersects only a finite number of sets in the locally finite covering $\{ V_{\mu} \}$ and, for any $a \in \A$, $\|a(t)\| \to 0$ as $t \to \infty$.
Let $\varepsilon >0$. There is a finite set $\lambda \in N(\Lambda)$
such that for $\mu \not\in \lambda $ we have 
\[\left\| a \sqrt{f_{\mu_p}}\right\|_{\A} < \frac{\varepsilon}{48 \ell^2}. \]

For  $\lambda \preccurlyeq \lambda', \lambda''$  and $m \ge n$, we have
\begin{align*}
 \| u_{a,\lambda'',m} - u_{a,\lambda',n}\|_{\A \widehat{\otimes} \A}&= 
 \| u_{a,\lambda'' \setminus \lambda,m} + u_{a,\lambda,m} - u_{a,\lambda' \setminus \lambda,n} -  u_{a,\lambda,n}\|_{\A \widehat{\otimes} \A}\\
& \le \| u_{a,\lambda'' \setminus \lambda,m}\|_{\A \widehat{\otimes} \A} + \| u_{a,\lambda' \setminus \lambda,n}\|_{\A \widehat{\otimes} \A} + \|u_{a,\lambda,m}-  u_{a,\lambda,n}\|_{\A \widehat{\otimes} \A}.
\end{align*}
By Lemma \ref{diagonal-element}, for $\tilde{\lambda} = \{\mu_1, \dots, \mu_m \}$,
\begin{align*}
\| u_{a,\tilde{\lambda},\tilde{n}} \|&= \left\|\sum_{\mu \in \tilde{\lambda}}\sum_{k=1}^{\tilde{n}}
a \sqrt{f_\mu}\;\; (\phi^\mu)^{-1}\left(\sqrt{u_k^\mu -u_{k-1}^\mu} \right) \otimes \sqrt{ f_\mu} \;\;(\phi^\mu)^{-1}\left( \sqrt{u_k^\mu -u_{k-1}^\mu} \right)\right\|\\
& \le \frac{1}{m \tilde{n}} \sum_{l=1}^{m}\sum_{k=1}^{\tilde{n}} 
\left\|\sum_{p=1}^{m} \sum_{i=1}^{\tilde{n}} \xi^{l p} \eta^{k i}
a \sqrt{f_{\mu_p}}\;(\phi^{\mu_p})^{-1}\left(\sqrt{u_i^{\mu_p} -u_{i-1}^{\mu_p}}\right) \right\| \\
& \;\;\;\;\; \times \left\|\sum_{q=1}^{m} \sum_{j=1}^{\tilde{n}} \xi^{-l q} \eta^{-k j} 
\sqrt{ f_{\mu_q}} \;(\phi^{\mu_q})^{-1}\left( \sqrt{u_j^{\mu_q} -u_{j-1}^{\mu_q}} \right)\right\|
\end{align*}
where $\xi$ is a primary $m$-th root of $1$ and $\eta$ is a primary $\tilde{n}$-th root of $1$ in $\CC$.
%%%%%%%%%%%%%%%%%%%%%%%%%%%%%%%%
By inequalities (\ref{prop-1-general}) and (\ref{prop-2-general}) from Lemma \ref{u_a_lambda_n},
\[
\| u_{a,\lambda'' \setminus \lambda,m}\|_{\A \widehat{\otimes} \A}  \le (4 \ell)^2 \max_{\mu \not\in \lambda}\left\| a \sqrt{f_{\mu}}\right\|_{\A} \le \frac{\varepsilon}{3}
\]
and 
\[ \| u_{a,\lambda' \setminus \lambda,n}\|_{\A \widehat{\otimes} \A} \le (4 \ell)^2 \max_{\mu \not\in \lambda}\left\| a \sqrt{f_{\mu}}\right\|_{\A} \le \frac{\varepsilon}{3}.
\]
By inequality (\ref{prop-3-general}) from Lemma \ref{u_a_lambda_n}, for $\lambda= \{\mu_1, \dots, \mu_N\}$,
\begin{align*}
~& \|u_{a,\lambda,m}-  u_{a,\lambda,n}\|_{\A \widehat{\otimes} \A}\\
~ & \le 4 \ell^2 \max_{1\le p\le N} \biggl( 10  \max_{n \le i \le m} \left\| \phi^{\mu_p}(a \sqrt{f_{\mu_p}}) u_i^{\mu_p} - \phi^{\mu_p} (a \sqrt{f_{\mu_p}}) \right\|_{C_0(V_{\mu_p},\tilde{A}_{\mu_p})} \;\left\| a \sqrt{f_{\mu_p}} \right\|_{\A}\\
~ & \quad + \frac{1}{2^{2n-1}} \left\|a \sqrt{f_{\mu_p}} \right\|_{\A}^2  \biggr)^{1/2}.
\end{align*}
By Lemma \ref{bai-in-C0(Omega,A)}, for every $\mu$,
\[ \left\| \phi^{\mu}(a \sqrt{f_{\mu}}) u_i^{\mu} - \phi^{\mu} (a \sqrt{f_{\mu}}) \right\|_{C_0(V_{\mu},\tilde{A}_{\mu})} \to 0 \;\; \text{as} \;\;
i \to \infty.
\]
Hence 
\[\|u_{a,\lambda,m}-  u_{a,\lambda,n}\|_{\A \widehat{\otimes} \A} \to 0 \;\; \text{as} \;\;
n,m \to \infty.
\]
Thus, in view of the completeness of $\A \widehat{\otimes} \A$,  for any $a \in \A$,
the net $u_{a,\lambda,n} $ converges in $\A \widehat{\otimes} \A$.
  \end{proof}

Now let us complete the proof of Theorem \ref{cont-field-strict-posit}.
Let us define $\rho: \A \rightarrow \A \widehat{\otimes} \A$ by setting, for every $a \in \A$,
\[
\rho(a) = \lim_\lambda \lim_{n \to \infty} \left(\sum_{\mu \in \lambda}\sum_{k=1}^{n}
a \sqrt{f_\mu}\;(\phi^\mu)^{-1}\left(\sqrt{u_k^\mu -u_{k-1}^\mu} \right) \otimes \sqrt{ f_\mu} \;(\phi^\mu)^{-1}\left( \sqrt{u_k^\mu -u_{k-1}^\mu} \right)   \right).
\]
By Lemma \ref{u_a_lambda_n-converges}, for any $a \in \A$,
the net $u_{a,\lambda,n} $ converges in $\A \widehat{\otimes} \A$.
It is  easy to see that  the map $\rho$ is well defined,   $\rho$  is a morphism of left  Banach  $\A$-modules and $\|\rho\|\le (4 \ell)^2$. By Lemma \ref{bai-in-C0(Omega,A)}, for every $a \in \A$,
\begin{align*}
(\pi_{\A} \circ \rho)(a) &=\lim_\lambda \lim_{n \to \infty} \left(\sum_{\mu \in \lambda}\sum_{k=1}^{n}
a f_\mu\;(\phi^\mu)^{-1}\left(u_k^\mu -u_{k-1}^\mu \right) \right)\\
&= \lim_\lambda \sum_{\mu \in \lambda} \lim_{n \to \infty} a f_\mu\;(\phi^\mu)^{-1}\left(u_n^\mu \right)\\
&= \lim_\lambda \sum_{\mu \in \lambda}  a f_\mu = a.
\end{align*}
Thus $\pi_{\A} \circ \rho = {\rm id}_{\A} $.

(ii) By Definition \ref{sigma-triv},
there is an open cover $\{W_\alpha \}$, $\alpha \in \M$, of $\Omega$ such that each $\mathcal{U}|_{W_\alpha}$  is trivial and, in addition,
there is an open cover $\{ B_j\}$ of cardinality $\ell$ of $\Omega$ such that $ \overline{B_j} \subset W_{\alpha(j)}$  for each $j =1, \dots, \ell,$ and some $\alpha(j) \in \M$.
By \cite[Lemma 2.1]{He3}, for any paracompact locally compact space $\Omega$ there exists an open cover $\U=\{U_\nu \}$  of relatively compact sets such that each point in $\Omega$ has a neighbourhood which intersects no more than three sets in $\U$. Consider an open cover
$\{ B_j \cap U_\nu: U_\nu \in \U, j = 1, \dots, \ell \}$ of $\Omega$.
Denote this cover by $\{V_{\mu} \}$. It is easy to see that $\{V_{\mu} \}$ is an open locally finite cover of $\Omega$ of order $3 \ell$. 
The rest of the proof of Part (ii) is similar to Part (i).
  \end{proof}

\begin{theorem}\label{Ax-strictly-positive-h-H2} Let $\Omega$ be a Hausdorff locally compact space with the topological dimension $\dim\Omega \le \ell$, for some $\ell \in \NN$,  let $\mathcal{U}=\{\Omega,A_{x},\Theta \}$ be a locally  trivial continuous field of $C^*$-algebras with strictly positive elements, and let the $C^*$-algebra $\A$ be defined by $\mathcal{U}$. Then the following conditions {\rm (i)} and {\rm (ii)}  are equivalent:\\
{\rm (i)} $\Omega$ is paracompact;\\
{\rm (ii)} $\A$ is left projective and $\mathcal{U}$ is a disjoint union of $\sigma$-locally  trivial continuous fields of $C^*$-algebras with strictly positive elements.\\
Moreover {\rm (i)} or {\rm (ii)} implies \\
{\rm (iii)}
${\H}^2(\A,X) = \{0 \}$ for any right annihilator Banach  $\A$-bimodule $X$.
\end{theorem}
\begin{proof}
By Theorem \ref{cont-field-strict-posit}, the fact that $\Omega$ is paracompact with the topological dimension $\dim\Omega \le \ell$ implies left projectivity of $\A$. By Remark \ref{parac-disjoint-union}, since
$\Omega$ is paracompact,  $\mathcal{U}$ is a disjoint union of $\sigma$-locally  trivial continuous fields of $C^*$-algebras. By Theorem \ref{Omega-paracompact},
conditions (ii) implies  paracompactness of $\Omega$.
Thus (i) $\Longleftrightarrow$ (ii).

By \cite[Proposition IV.2.10(I)]{He4}, $\A$ is left projective if and only if $ {\H}^2(\A,X) = \{0 \}$ for any right annihilator Banach  $\A$-bimodule $X$ and so (ii)  $\Longrightarrow$ (iii). 
  \end{proof}

%%%%%%%%%%%%%%%%%%%%%%%%%%%

DAVID CUSHING,
School of Mathematics and Statistics, Newcastle University,
 NE\textup{1} \textup{7}RU, U.K.~~
e-mail\textup{: \texttt{david.cushing@newcastle.ac.uk}}\\

ZINAIDA A. LYKOVA,
School of Mathematics and Statistics, Newcastle University,
 NE\textup{1} \textup{7}RU, U.K.~~
e-mail\textup{: \texttt{Z.A.Lykova@newcastle.ac.uk}}\\


\begin{thebibliography}{30}

%%%%%%%%%
\bibitem{AaKa}  J. F. Aarnes and R. V. Kadison. Pure states and approximate identities. \textit{ Proc. Amer. Math. Soc.} \textbf{ 21}  (1969), 749--752. 

%%%%%%%%
\bibitem{AP}   A. V. Arkhangel'ski\u\i \; and V. I. Ponomarev.
\textit{ Fundamentals of general topology through problems and exercises.}
 Nauka, Moscow, 1974 (in Russian); Mathematics and its Applications. 
D. Reidel Publishing Co., Dordrecht-Boston, Mass., 1984.


%%%%%%%
\bibitem{Ari00} O. Yu. Aristov. Homological dimensions of $C^*$-algebras. In
\textit{ Topological homology. Helemskii's Moscow seminar.}  Ed. Helemskii, A. Ya., Huntington, NY: Nova Science Publishers, 2000, pp. 39--55.

%%%%%
\bibitem{BDL} W. G. Bade,  H. G. Dales and Z. A. Lykova.
 Algebraic and strong splittings of
extensions of Banach algebras, \textit{ Mem. Amer. Math. 
Soc.} 	\textbf{ 137} (1999), no. 656, 113pp.

%%%%
\bibitem{CF-T} P. C. Curtis, Jr. and A. Figa-Talamanca.
 Factorization theorems for Banach algebras. In \textit{
 Function Algebras}. Proc. Internat. Sympos. on Function Algebras, 
Tulane University, 1965, Scott-Foresman, Chicago, Ill. 1966, pp. 169--185.


%%%%
\bibitem{DDPR} H. D. Dales, M. Daws, H. L. Pham and P. Ramsden. Multi-norms and the injectivity of $L^p(G)$.   
arXiv:1101.4320v1 [math.FA] 22 Jan 2011, pp. 29.

\bibitem{DP} H. G. Dales and M. E. Polyakov. Homological properties of modules over group algebras. \textit{Proceedings of the London Mathematical Society} (3), \textbf{89} (2004), 390--426.

\bibitem{DW}    R. S. Doran and J. Wichmann. 
 \textit{ Approximate Identities and Factorization in Banach Modules}. Springer-Verlag, 1979.

%%%%%
\bibitem{En}    R. Engelking. 
 \textit{ General Topology}. Sigma Series in Pure Mathematics, 6,
 Heldermann Verlag, Berlin, 1989.
%%%%%
\bibitem{Di}    J. Dixmier. 
  \textit{ Les $C^*$-alg${\rm \grave{e}}$bres et leurs repr${\rm
\acute{e}}$sentations}. 
Gauthier-Villars, Paris, 1969.


\bibitem{He3}   A. Ya. Helemskii. 
A description of relatively projective ideals in the algebras $C(\Omega)$.
  \textit{ Doklad. Akad. Nauk SSSR} (6) \textbf{195} (1970),
1286-1289  (in Russian); 
\textit{ Soviet Math. Dokl.}  \textbf{1} (1970), 1680--1683.

%%%%
\bibitem{He4}  A. Ya. Helemskii. \textit{  The Homology of Banach and Topological
Algebras}.  Moscow Univ. Press,  1986 (in Russian); Kluwer Academic
Publishers, 1989 (in English).


\bibitem{Hew} E. Hewitt. The ranges of certain convolution
operators.   \textit{ Math. Scand.} \textbf{15} (1964), 147--155.


\bibitem{Jo1}   B. E. Johnson. 
 Cohomology of Banach algebras, \textit{ Mem. Amer. Math. Soc.}
\textbf{127} (1972).

%%%%%
\bibitem{Jo2}   B. E. Johnson.  The Wedderburn decomposition of Banach algebras with finite-dimensional radical.  \textit{  American J. Math. }
\textbf{90} (1968), 866--876.

%%%%
\bibitem{MJo}   M. S. G. Jones.  The projectivity of some Banach algebras, \textit{  MPhil Dissertation} (2006), Newcastle University, U.K.


%%%%
\bibitem{Ke}   J. L. Kelley. \textit{ General topology}. 
(Van Nostrand, Princeton, 1955); with \textit{ Addition}
 by A.V. Arkhangel'skii, Nauka, Moscow, 1981 (in Russian).

\bibitem{KirWas} E. Kirchberg and S. Wassermann. Operations on continuous bundles of $C^*$-algebras. \textit{Math. Ann.} \textbf{303} (1995), 677--697.

%%%%
\bibitem{Ly2}   Z. A. Lykova. On homological characteristics 
of operator algebras. \textit{ Vest. Mosk. Univ. ser.
mat. mech.} \textbf{1} (1986), 8--13 (in Russian); \textit{Moscow Univ. 
Math. Bull.} (3) \textbf{41} (1986),
 10--15.

\bibitem{Ly02}   Z. A. Lykova. Relations between the homologies of $C^*$-algebras  and their commutative $C^*$-subalgebras. \textit{ Math. Proc. Camb. Phil. Soc.} (2) \textbf{ 132}(2002), 155--168.

\bibitem{PhR}  J. Phillips and I. Raeburn. Central cohomology of $C^*$-algebras. \textit{J London Math. Soc.} (2) \textbf{28} (1983), 363--375.

\bibitem{Ra}  G. Racher. On the projectivity and flatness of some group modules. In \textit{Banach Algebras 2009.}
 Banach Center Publications \textbf{91} (2010), Warszawa, pp. 315--325.

\bibitem{Ri}    C. E. Rickart. 
  \textit{ General theory of Banach algebras}. Van Nostrand, Princeton, 1960.

\bibitem{Se00}   Yu. V. Selivanov. Coretraction problems and homological properties of Banach Algebras. In \textit{ Topological homology. Helemskii's Moscow seminar.} Ed. Helemskii, A. Ya., Huntington, NY: Nova Science Publishers, 2000, pp. 145--199.

%%%%

\bibitem{Se2} Yu. V. Selivanov. Biprojective Banach
algebras, their structure, cohomology, and relation with nuclear operators.
  \textit{ Funct. anal. i pril.} (1)  \textbf{10} (1976),   89--90 (in Russian); 
\textit{  Functional Anal. Appl.}  \textbf{10} (1976), 78--79.

\bibitem{Se95} Yu. V. Selivanov. Cohomology of biflat  Banach
algebras with coefficients in dual bimodules. 
\textit{  Functional Anal. Appl.}  \textbf{25} (1995), no. 4, 289--291.

\bibitem{Sh}    H. Schaefer. 
  \textit{Topological vector spaces}. Macmillan, New York, 1966.

\bibitem{Wh}  M. C. White. Injective modules for
uniform algebras.
  \textit{Proc. London Math. Soc.} (3) \textbf{73} (1996), 155--184.\\
\end{thebibliography}
\end{document}